\documentclass[12pt]{article}
\usepackage{amsmath,amsfonts,amssymb,amsthm}
\newtheorem{theo}{Theorem}[section]
\newtheorem{prop}[theo]{Proposition}
\newtheorem{coro}[theo]{Corollary}
\newtheorem*{lem}{Lemma}
\numberwithin{equation}{section}
\newcommand\la{\lambda}
\newcommand\al{\alpha}
\newcommand\be{\beta}
\newcommand\ta{\theta}

\newcommand\npbi[3]{{\genfrac{\langle}{\rangle}{0pt}{}{#1}{#2}}_{#3}}

\title{Class expansion of some symmetric functions
\\ in Jucys-Murphy elements}
\author{Michel Lassalle\\
\small Centre National de la Recherche Scientifique\\[-0.8ex]
\small Institut Gaspard-Monge, Universit\'e de Marne-la-Vall\'ee\\[-0.8ex]
\small 77454 Marne-la-Vall\'ee Cedex, France\\[-0.8ex]
\small \texttt{lassalle@univ-mlv.fr}\\[-0.8ex]
\small \texttt{http://igm.univ-mlv.fr/{\textasciitilde}lassalle}
}
\date{}
\begin{document}
\maketitle

\begin{abstract}
We present a method to compute the class expansion of a symmetric function in the Jucys-Murphy elements of the symmetric group. We apply this method to one-row Hall-Littlewood symmetric functions, which interpolate between power sums and complete symmetric functions.
\end{abstract}

\section{Introduction}

Let $S_n$ be the group of permutations of $n$ letters, $\mathbb{C}[S_n]$ its group algebra and $\mathcal{Z}_n$ the center of $\mathbb{C}[S_n]$. Given a partition $\mu$ with weight $n$, denote by $C_\mu$ the conjugacy class of permutations having cycle-type $\mu$, viewed as the formal sum of its elements. These classes form a basis of $\mathcal{Z}_n$. 

For $i=1,\ldots,n$ the Jucys-Murphy elements $J_i$ are defined by $J_i=\sum_{j<i} (ji)$, where $(ji)$ is a transposition. These commutative elements were introduced independently in~\cite{Ju}  and~\cite{Mu1,Mu2}. They do not belong to $\mathcal{Z}_n$. However Jucys and Murphy proved, by different means, that $\mathcal{Z}_n$ coincides with the algebra of symmetric functions in the $J_i$'s. 

Given a symmetric function $f$, it is therefore a natural problem to study the class expansion
\begin{equation*}
f(J_1,\ldots,J_n)=\sum_{|\mu|=n}  a_{\mu}(n) \, C_\mu, 
\end{equation*}
i.e. the development of its specialization $f(J_1,\ldots,J_n)$ in terms of the basis $C_\mu$ of $\mathcal{Z}_n$. The purpose of this paper is to present a general method to compute such an expansion.

This problem had been solved before by Jucys~\cite{Ju} for $f=e_k$,  the elementary symmetric function, and by Lascoux and Thibon~\cite{LT} for $f=p_k$, the power sum symmetric function. Our method provides a new proof for these classical results, but also allows to handle many new cases. 

In this paper we consider $f=h_k$, the complete symmetric function and $f=P_k(z)$ the one-row Hall-Littlewood symmetric function, which interpolates between $h_k$ and $p_k$. But, among others, our method also works with the product $f=h_ke_l$ or with $f=s_{(k,1^l)}$, the Schur function associated  with hooks.

It should be emphasized that our method is not performed in the symmetric algebra $\mathbb{C}[S_n]$, but rather in the shifted symmetric algebra~\cite{KO,OO}. Actually given a partition $\la$ and $\chi^\la$ the character of the corresponding irreducible representation, by a celebrated result of Jucys~\cite[eq. (12)]{Ju} we have 
\[f(J_1,\ldots,J_n)\, \chi^\la=f(A_\la)\, \chi^\la,\]
where $A_\la$ denotes the alphabet of ``contents'' of $\la$ and $\chi^\la$ stands for the central element $\sum_{\sigma\in S_n} \chi^\la(\sigma) \sigma$.

This fundamental property allows us to translate the class expansion of the central element $f(J_1,\ldots,J_n)$ in terms of the content evaluation $f(A_\la)$.  Actually if we define the central character $\ta^\la_\mu$ by $C_\mu\, \chi^\la=\ta^\la_\mu\, \chi^\la$, we may write \textit{equivalently}
\begin{equation*}
f(J_1,\ldots,J_n)=\sum_{|\mu|=n}  a_{\mu}(n) \, C_\mu\quad
\textrm{or} \qquad
f(A_\la)=\sum_{|\mu|=n}  a_{\mu}(n) \, \ta^\la_\mu.
\end{equation*}
In this paper we consider the second equality, which connects two shifted symmetric functions and can be studied by analytic means.

This method is indirect but presents the advantage of having a very natural extension in the framework of Jack polynomials. In that context the symmetric algebra and the Jucys-Murphy elements have not yet been generalized, but the algebra of $\al$-shifted symmetric functions is very well known. 

The paper is organized as follows. Section 2 is devoted to general facts about the symmetric group. Section 3 recalls results about the transition measure. Section 4 presents our tools and a summary of our method. The latter is used in Sections 5 and 6 to recover the classical results of Jucys~\cite{Ju} and Lascoux-Thibon~\cite{LT}. Sections 7 and 8 are respectively devoted to the new cases of complete symmetric functions and Hall-Littlewood symmetric functions. The generating functions associated with these class expansions are considered in Section 9. An extension of our method in the framework of Jack polynomials is briefly sketched at the end.

\section{Generalities and notations}

We recall some notions about the symmetric group and its representations, referring the reader to~\cite{CST} for a detailed account. In this paper $n$ is an arbitrary positive integer.

\subsection{Permutations and partitions}

A partition $\la= (\la_1,...,\la_r)$
is a finite weakly decreasing
sequence of nonnegative integers, called parts. The number
$l(\la)$ of positive parts is called the length of
$\la$, and $|\la| = \sum_{i = 1}^{r} \la_i$
the weight of $\la$. For any integer $i\geq1$,
$m_i(\la) = \textrm{card} \{j: \la_j  = i\}$
is the multiplicity of the part $i$ in $\la$. We write
$\la= (1^{m_1(\la)},2^{m_2(\la)},3^{m_3(\la)},\ldots)$, and $\la \vdash n$ for $|\la|=n$.

We denote by $\la^{\prime}$ the partition conjugate to $\la$, with parts given by $\la^{\prime}_i=\sum_{j \ge i} m_j(\la)$. We identify $\la$ with its Ferrers diagram  
$\{ (i,j) : 1 \le i \, \le l(\la), 1 \le j \le {\la}_{i} \}$. We set
\[z_\la  = \prod_{i \ge  1} i^{m_i(\lambda)} m_i(\lambda) !,\qquad
H_\la=\prod_{(i,j)\in \la} (\la_i+\la^\prime_j-i-j+1).\]

For any partition $\la$ and any integer $1 \le i \le l(\la)+1$, we denote by $\la^{(i)}$ the partition $\mu$ (if it exists) such that $\mu_j=\la_j$ for $j\neq i$ and $\mu_i=\la_i +1$. Similarly for any integer $1 \le i \le l(\la)$, we denote by $\la_{(i)}$ the partition $\nu$ (if it exists) such that $\nu_j=\la_j$ for $j\neq i$ and $\nu_i=\la_i -1$. 

Given some positive integers $p,q$ and a partition $\rho$, we denote by $\rho \setminus (p) \cup (q)$ the partition (if it exists) obtained by removing a part $p$ and adding a part $q$ to $\rho$. We denote by $\overline{\rho}$ the partition obtained by erasing all parts $1$ of $\rho$. Thus $m_1(\overline{\rho}) = 0$ and $\rho = \overline{\rho}\cup 1^{|\rho|-|\overline{\rho}|}$. Conversely we denote by $\tilde{\rho}$ the partition obtained by adding parts $1$ to $\rho$ up to the weight $n$. Thus $m_1(\tilde{\rho}) = m_1(\rho)+n-|\rho|$ and $\tilde{\rho}=\rho\cup 1^{n-|\rho|}$.

Let $S_n$ be the group of permutations of $n$ letters, $\mathbb{C}[S_n]$ its group algebra and $\mathcal{Z}_n$ the center of $\mathbb{C}[S_n]$. Each permutation $\sigma \in S_n$ factorizes uniquely as a product of disjoint cycles, whose respective lengths are ordered such as to form a partition $\mu=(\mu_1,\ldots,\mu_r)$ with weight $n$. This partition, called the cycle-type of $\sigma$, determines each permutation up to conjugacy in $S_n$. Given a partition $\mu \vdash n$, we denote by $C_\mu$ the conjugacy class of permutations having cycle-type $\mu$. 

We view any central function $\chi$ on $S_n$ as the formal sum $\sum_{\sigma\in S_n}\chi(\sigma) \sigma \in \mathcal{Z}_n$. We identify each $C_\mu$ with its characteristic function, hence with the formal sum of its elements. The set $\{C_\mu, \mu \vdash n\}$ forms a basis of $\mathcal{Z}_n$. 

\subsection{Symmetric functions}

Let $A=\{a_1,a_2,a_3,\ldots\}$ a (possibly infinite) set of
independent indeterminates, called an alphabet. The generating functions
\begin{equation*}
E_z(A)=\prod_{a\in A} (1 +za) =\sum_{k\geq0} z^k\, e_k(A),\qquad
H_z(A)=\prod_{a\in A}  (1-za)^{-1} = \sum_{k\geq0} z^k\, h_k(A)
\end{equation*}
define symmetric functions known as respectively elementary and complete. The power sum symmetric functions
are defined by $p_{k}(A)=\sum_{i \ge 1} a_i^k$. For any partition $\mu$, we define functions $e_\mu$, 
$h_\mu$ or $p_\mu$ by
\[f_{\mu}= \prod_{i=1}^{l(\mu)}f_{\mu_{i}}=\prod_{k\geq1}f_k^{m_{k}(\mu)},\]
where $f_{k}$ stands for $e_k$, $h_k$  or $p_k$. 

When $A$ is infinite, each of the three sets of functions $e_k$, 
$h_k$ or $p_k$ forms an algebraic basis of
$\mathbb{S}$, the symmetric algebra with coefficients in $\mathbf{R}$. Each of the sets of functions $e_\mu$, $h_\mu$, $p_\mu$ is a linear basis of this algebra. Two other linear bases are formed by the Schur functions $s_\la$ and by the monomial symmetric functions $m_{\la}$, defined as the sum of all distinct monomials whose exponent is a permutation of $\la$.

\subsection{Shifted symmetric functions}

Although the theory of symmetric functions goes back to the early 19th century, shifted symmetric functions are quite recent. They were introduced and studied in~\cite{KO,OO}.

Being given a finite alphabet  $A=\{a_1,a_2,\ldots,a_r\}$, a polynomial in $A$ is ``shifted symmetric'' if it is symmetric in the shifted variables $a_i-i$. When $A=\{a_1,a_2,a_3,\ldots\}$ is infinite, in analogy with symmetric functions, 
a ``shifted symmetric function'' $f$ is a family $\{f_i, i\ge 1\}$ such that  $f_i$ is a shifted symmetric polynomial in $(a_1,a_2,\ldots,a_i)$, together with the stability property $f_{j}(a_1,a_2,\ldots,a_i,0,\ldots,0)=f_i(a_1,a_2,\ldots,a_i)$ whenever $j\ge i$.  

This defines $\mathbb{S}^{\ast}$, the shifted symmetric algebra with coefficients in $\mathbf{R}$, which is algebraically generated by the ``shifted power sums''
\[p_k^{*}(A)=\sum_{i\ge 
1}\Big((a_i-i+1)_k-(-i+1)_k\Big).\]
Here for an indeterminate $z$ and any positive integer $p$, the \textit{falling} factorial
\[(z)_p = z(z-1) \ldots (z-p+1)=\sum_{i=1}^p s(p,i) \,z^i,\]
is the generating function of the Stirling numbers of the first kind $s(p,i)$. Conversely 
\[z^p=\sum_{i=1}^p S(p,i)(z)_i\]
defines the Stirling numbers of the second kind $S(p,i)$.

An element $f\in \mathbb{S}^{\ast}$ may be evaluated at any sequence 
$(a_1,a_2,\ldots)$ with finitely many non zero terms, hence at any partition $\la$. Moreover by analyticity, $f$ is entirely determined by its restriction $f(\la)$ to partitions. This identification is usually performed and $\mathbb{S}^{\ast}$ is considered as a function 
algebra on the set of partitions.

\subsection{Contents}

Given a partition $\la$, the content of any node $(i,j) \in \la$ is defined as $j-i$. Denote by $A_\la=\left\{j-i,\, (i,j) \in \la \right\}$ the finite alphabet of the contents of $\la$. The symmetric algebra $\mathbb{S}[A_\la]$ is generated by the power sums
\[p_k(A_\la) = \sum_{(i,j) \in \la} (j-i)^k=
\sum_{i=1}^{l(\la)} \sum_{j=1}^{\la_i} (j-i)^k.\]

It is well known~\cite{KO,OO} that the quantities $p_k(A_\la)$ are shifted symmetric polynomials of $\la$. Indeed for any integer $k\ge 1$, applying the identity $r(z)_{r-1}=(z+1)_{r}-(z)_{r}$, we have
\begin{align*}
p_k(A_\la)&= \sum_{r=1}^k \sum_{(i,j)\in \la} S(k,r)\,
(j-i)_{r}\\
&= \sum_{r=1}^k \frac{S(k,r)}{r+1} \sum_{i= 1}^{l(\la)} 
\Big((\la_i-i+1)_{r+1}-(-i+1)_{r+1}\Big)\\
&=\sum_{r=1}^k  \frac{S(k,r)}{r+1}\, p_{r+1}^{*}(\la).
\end{align*} 

As a straightforward consequence, the shifted symmetric algebra $\mathbb{S}^{\ast}$ is algebraically generated by the functions $p_k(A_\la), k\ge 1$ together with $p_1^{*}(\la)=|\la|$. The latter corresponds to the cardinal of the alphabet $A_\la$. 

In other words, any shifted symmetric function may be written \textit{in a unique way} $f(A_\la)$, with $f \in \mathbf{R}[\textrm{card},p_1,p_2,p_3,\ldots]$. As mentioned in~\cite[Proposition 2.4]{Ol}, this fact was already known to Kerov.

\subsection{Representations}

The irreducible representations of $S_n$ and their characters are labelled by partitions $\la \vdash n$. Given such a partition, we denote by $\chi^\la$ the corresponding irreducible character. We write ${\chi}^\la_\mu$ for its value $\chi^\la(\sigma)$ at any permutation $\sigma$ of cycle-type $\mu$, and $\textrm{dim}\,\la=\chi^\la_{1^n}$ for the dimension of the representation $\la$. The latter is given by
\[\textrm{dim}\,\la=  \frac{n!}{H_\la}= \frac{n!}{\prod_{i=1}^{l(\la)} (\la_i+l(\la)-i)!} \prod_{1\le i<j\le l(\la)}(\la_i-\la_j+j-i).\]

We denote
\[\hat{\chi}^\la_\mu=\frac{\chi^\la_\mu}{\textrm{dim}\,\la}, \qquad
\ta^\la_\mu=\frac{H_\la}{z_\mu} \,\chi^\la_\mu=\frac{n!}{z_\mu} \,\hat{\chi}^\la_\mu,\]
respectively the normalized character and central character of the representation $\la$. In $\mathcal{Z}_n$ we have the decompositions
\[\chi^\la=\sum_{\mu \vdash n} \chi^\la_\mu C_\mu, \qquad
C_\mu=\sum_{\la \vdash n} \frac{\chi^\la_\mu}{z_\mu} \chi^\la.\]
The family $\{\chi^\la/H_\la, \la \vdash n\}$ forms a basis of orthogonal idempotents in $\mathcal{Z}_n$, which yields
\[C_\mu\, \chi^\la =\sum_{\rho \vdash n} \frac{\chi^\rho_\mu}{z_\mu} \chi^\rho \chi^\la=\ta^\la_\mu\, \chi^\la.\]

\subsection{Jucys-Murphy elements}

For $1\le i\le n$ the Jucys-Murphy elements $J_i$ are defined by \[J_i=\sum_{1\le j<i} (ji),\]
where $(ji)$ denotes a transposition. These elements generate a maximal commutative subalgebra of $\mathbb{C}[S_n]$. Jucys~\cite{Ju} and Murphy~\cite{Mu2} proved the following fundamental property.
 
\begin{theo}
The center $\mathcal{Z}_n$ is formed of the elements $f(J_1,\ldots,J_n)$, with $f$ a symmetric function. These elements act on irreducible characters by
\[f(J_1,\ldots,J_n)\, \chi^\la=f(A_\la)\, \chi^\la.\]
\end{theo}

This result has the following important consequence.
\begin{coro}
For any symmetric functions $f,g$ (which may depend polynomially on $n$), the following statements are equivalent:
\begin{align*}
\mathrm{(i)}\qquad &f=g,\\
\mathrm{(ii)} \qquad &f(J_1,\ldots,J_n)=g(J_1,\ldots,J_n)
\quad \textrm{for} \; \textrm{any}\; n\ge 1,\\
\mathrm{(iii)}\qquad  &f(A_\la)=g(A_\la)\quad \textrm{for} \; \textrm{any}\; \textrm{partition}\; \la.
\end{align*}
\end{coro}
\begin{proof} In view of Theorem 2.1 the implications $(i) \Rightarrow (ii) \Rightarrow (iii)$ are obvious. The implication $(iii) \Rightarrow (i)$ is a consequence of the unicity result stated at the end of Section 2.4 (the polynomial dependence on the cardinal is there crucial).
\end{proof}

Now given some symmetric function $f\in \mathbb{S}$, consider
\begin{equation}
f(J_1,\ldots,J_n)=\sum_{|\mu|=n}  a_{\mu}(n) \, C_\mu, 
\end{equation}
the class expansion of its Jucys-Murphy specialization $f(J_1,\ldots,J_n)\in \mathcal{Z}_n$. By taking eigenvalues associated to $\chi^\la$ we get
\begin{equation}
f(A_\la)=\sum_{|\mu|=n}  a_{\mu}(n) \, \theta^\la_\mu. 
\end{equation}
for any partition $\la \vdash n$. The central character $\ta^\la_\mu$ has been studied by several authors~\cite{KO,OO,Co,La}. It is well known that it can be extended to a shifted symmetric function of $\la$. Thus the previous equality holds in $\mathbb{S}^{\ast}$. We may \textit{equivalently} study this decomposition in $\mathbb{S}^{\ast}$, rather than the original one in $\mathcal{Z}_n$.

\subsection{Inverse problem}

We may also consider the inverse problem, and look for an expression of each class $C_\mu$ as the Jucys-Murphy specialization
\begin{equation*}
C_\mu=f_\mu(J_1,\ldots,J_n) 
\end{equation*}
of some symmetric function $f_\mu$ (depending polynomially on $n$). 
This amounts to write the corresponding central character as the content evaluation of $f_\mu$, namely
\[\theta^\la_\mu=f_\mu(A_\la).\]
Moreover it is equivalent to write (2.1), (2.2) or
\begin{equation}
f=\sum_{|\mu|=n}  a_{\mu}(n) \, f_\mu. 
\end{equation}
In other words, the symmetric functions $f_\mu$ form a basis of $\mathbb{S}$ (depending polynomially on $n$)~\cite{KO,Co}.

The functions $f_\mu$ have been made explicit in~\cite{La}, up to a constant factor. Actually the central character $\theta^\la_\mu$ may be written as
\[\theta^\la_\mu=\frac{n!}{z_\mu} \,\hat{\chi}^\la_\mu
=\frac{(n)_{|\overline{\mu}|}}{z_{\overline{\mu}}}\,\hat{\chi}^\la_\mu =z_{\overline{\mu}}^{-1}g_\mu(A_\la),\]
where $g_\mu$ is some symmetric function (depending polynomially on $n$), explicitly given in~\cite{La} in terms of auxiliary symmetric functions.

In other words we have $f_\mu=z_{\overline{\mu}}^{-1}g_\mu$. Tables giving  $g_\mu$ for $|\mu|-l(\mu)\le 14$ are available on a web page~\cite{W}.

\subsection{Dependence on $n$}

The notion of partial permutation of $\{1,\dots,n\}$ has been introduced in~\cite{IK}. It leads to define an abstract algebra $\mathbb{B}$, a basis of which is formed by elements $B_\rho$ indexed by all partitions $\rho$. 

There is an isomorphism $\iota$ between this algebra and the shifted symmetric algebra $\mathbb{S}^\ast$, which may be described as follows~\cite[Theorem 9.1]{IK}. For any partition $\rho$ with $|\rho|\le n$ we have
\begin{equation*}
\begin{split}
\iota(B_\rho)(\la)&=
\frac{(n)_{|\rho|}}{z_\rho}\,\hat{\chi}^\la_{\tilde{\rho}}\\
&=\binom{n-|\rho|+m_1(\rho)}{m_1(\rho)}\, \frac{n!}{z_{\tilde{\rho}}}\,\hat{\chi}^\la_{\tilde{\rho}}\\
&=\binom{n-|\overline{\rho}|}{m_1(\rho)}\, \theta^\la_{\tilde{\rho}},
\end{split}
\end{equation*}
with $\la \vdash n$. For $|\rho|> n$ we have $\iota(B_\rho)(\lambda)=0$.

This isomorphism implies that the decomposition (2.2) of the shifted symmetric function $f(A_\la)$ takes the form
\[f(A_\la)=\sum_{|\rho|\le n}c_\rho \, \binom{n-|\overline{\rho}|}{m_1(\rho)}\, \theta^\la_{\tilde{\rho}},\]
the coefficients $c_\rho$ being independent of $n$. Equivalently (2.1) may be written as
\[f(J_1,\ldots,J_n)=\sum_{|\rho|\le n}c_\rho \, \binom{n-|\overline{\rho}|}{m_1(\rho)} \, C_{\tilde{\rho}}.\]
In other words $a_{\mu}(n)$ may be written as
\[a_{\mu}(n)=\sum_{\rho}c_\rho \, \binom{n-|\overline{\rho}|}{m_1(\rho)},\]
summed over partitions $\rho$ satisfying $\overline{\rho}=\overline{\mu}$, or equivalently $\tilde{\rho}=\mu$.

\section{The transition measure}

Given a partition $\la\vdash n$, the transition measure $\omega_\la$ is a probability measure on the real line, studied by Kerov~\cite{K1,K2} and others.

For any $i=1,\ldots, l(\la)+1$ we define the transition probabilities
\[c_i(\la)= \frac{H_\la}{H_{\la^{(i)}}}=\frac{1}{n+1} \,\frac{\textrm{dim}\,\la^{(i)}}{\textrm{dim}\,\la},\]
if  the partition $\la^{(i)}$ exists, and $0$ otherwise. We have easily
\[c_i(\la) = \frac {1}{\la_i+l(\la)-i+2}
\prod_{\begin{subarray}{c}j=1 \\ j \neq i\end{subarray}}^{l(\la)+1} 
\frac{\la_i-\la_j+j-i+1}
{\la_i-\la_j+j-i}.\]

We consider the discrete measure 
\[\omega_\la=\sum_{i=1}^{l(\la)+1} c_i(\la)\, \delta_{\la_i-i+1},\]
where $\delta_{u}$ is the Dirac measure at $u$. It is a probability measure, supported by the points $\la_i-i+1$ such that $\la^{(i)}$ exists.

We denote the moments of $\omega_\la$ by
\begin{equation}
\sigma_k(\la)=\sum_{i=1}^{l(\la)+1} c_i(\la)\,(\la_i-i+1)^k.
\end{equation}
The moment generating series of $\omega_\la$ is given by
\begin{equation}
\mathcal{M}_\la(z)= \sum_{k\ge 0} \sigma_k(\la) z^{-k-1}=\sum_{i=1}^{l(\la)+1} 
 \frac{c_i(\la)}{z-\la_i+i-1}.
\end{equation}

In the more general context of Jack polynomials, we have shown that $\mathcal{M}_\la$ may be alternatively written
\begin{equation}
\mathcal{M}_\la(z)= z^{-1} \,\frac{C_\la(-z)}{C_\la(-z-1)}\, \frac{C_\la(-z)}{C_\la(-z+1)},
\end{equation} 
where $C_\la(z)$ denotes the content polynomial
\[C_\la(z)=\prod_{(i,j)\in \la}(z+j-i).\]
A proof is obtained by setting $\alpha=1$ in a more general result, proved in~\cite[Theorem 8.1, p. 3470]{La1} by using Lagrange interpolation.

We may identify the developments of (3.2) and (3.3) in descending powers of $z$. By~\cite[Corollary 5.2, p. 3464]{La1} (written for $y=-1$) we obtain 
\begin{equation*}\sigma_k(\la)=f_k(A_\la),
\end{equation*}
where
\begin{equation}
f_k=\sum_{\begin{subarray}{c}q,r \ge 0 \\ q+2r \le k\end{subarray}} \sum_{s=0}^{\mathrm{min}(r,k-2r)}
\binom{n+r-1}{r-s} \sum_{|\mu| = k-2r} 
\npbi{\mu}{q}{s}\,z_{\mu}^{-1}p_\mu
\end{equation}
is a symmetric function depending on $n$. Here $\npbi{\mu}{q}{s}$ is some positive integer explicitly known (see~\cite{La2},~\cite[p. 3459]{La1} or~\cite[p. 392]{La}), in particular by a generating function.

In view of Section 2.4, the moments $\sigma_k(\la)$ are shifted symmetric functions. In this paper we shall mainly need the following elementary values
\begin{equation}
\sigma_0(\la)=1,\qquad \sigma_1(\la)=0,\qquad \sigma_2(\la)=n,
\qquad \sigma_3(\la)=2p_1(A_\la).
\end{equation}
But we may also mention 
\begin{equation}
\begin{split}
\sigma_4(\la)&=3p_2(A_\la)+\binom{n+1}{2},\\
\sigma_5(\la)&=4p_3(A_\la)+2(n+1)p_1(A_\la),\\
\sigma_6(\la)&=5p_4(A_\la)+3(n+1)p_2(A_\la)+2p_2(A_\la)+2p_1^2(A_\la)+\binom{n+2}{3}.
\end{split}
\end{equation}

Since the moments may be written $\sigma_k(\la)=f_k(A_\la)$ with $f_k$ a symmetric function depending on $n$, we may also consider the central element 
$$M^{(k)}_n=f_k(J_1,\ldots,J_n).$$
Biane~\cite{B1,B2} has shown that $M^{(k)}_n=\pi(J_{n+1}^k)$, with $\pi$ the orthogonal projection of $\mathbb{C}[S_{n+1}]$ onto $\mathbb{C}[S_n]$. By Jucys' result, for any $\la \vdash n$ we have
$$M^{(k)}_n \,\chi^\la=\sigma_k(\la)\,\chi^\la.$$
It is natural to study the equivalent expansions
\begin{equation*}
\begin{split}
M^{(k)}_n=\sum_{|\mu|=n} s^{(k)}_\mu(n) \, C_{\mu}&= \sum_{\rho}\mathbf{s}^{(k)}_\rho \, \binom{n-|\overline{\rho}|}{m_1(\rho)} \, C_{\tilde{\rho}},\\
\sigma_k(\la)=\sum_{|\mu|=n} s^{(k)}_\mu(n) \, \theta^\la_{\mu}&= \sum_{\rho}\mathbf{s}^{(k)}_\rho \, \binom{n-|\overline{\rho}|}{m_1(\rho)} \, \theta^\la_{\tilde{\rho}}.
\end{split}
\end{equation*}
We shall give them explicitly at the end of Section 6.

\section{Tools and method}

In this paper we make a crucial use of some linear relations between central characters. The following auxiliary material is needed.

\subsection{Differential operators}

In the space $\mathbf{R}^N$ of $N$ variables $(x_1,\ldots,x_N)$,
for any integer $k\ge 0$ we introduce the differential operators
\begin{equation*}
\begin{split}
E_k&=\sum_{i=1}^N x_i^k\, \frac{\partial}{\partial x_i},\\
D_k&=\frac{1}{2}\sum_{i=1}^N x_i^k\, \frac{\partial^2}{\partial x_i^2}
+ \sum_{\begin{subarray}{c}i,j=1\\i\neq j\end{subarray}}^N 
\frac{x_i^k}{x_i-x_j}\frac{\partial}{\partial x_i}.
\end{split}
\end{equation*} 
It is not difficult to check that $D_0=[E_0,D_1]$ and $2D_1=[E_0,D_2]$. The following result is proved by an easy induction on $N$.
\begin{lem} For any integer $r\ge 2$, we have
\[2\sum_{\begin{subarray}{c}i,j=1\\i\neq j\end{subarray}}^N 
\frac{x_i^r}{x_i-x_j}=\sum_{i=1}^{r-2}p_i\,p_{r-i-1}+(2N-r)p_{r-1}.\]
\end{lem}
For $r=1$ the left-hand side is obviously $N(N-1)$. After some easy but tedious computation, for any integer $k\ge 2$ and any partition $\mu$, this lemma yields
\begin{equation}
\begin{split}
2D_k\,p_\mu&=
\sum_{r, s\ge 1} rsm_r(\mu)(m_s(\mu)-\delta_{rs}) 
\,p_{\mu \setminus (r,s)\cup (r+s+k-2)}\\
&+\sum_{r\ge 1} rm_r(\mu) \sum_{i=1}^{r+k-3}
\,p_{\mu \setminus (r)\cup (i,r-i+k-2)}
+(2N-k)\sum_{r\ge 1}rm_r(\mu)\,p_{\mu \setminus (r)\cup (r+k-2)}.
\end{split}
\end{equation}
Some attention is needed for $k=1$, where we get
\begin{equation}
\begin{split}
2D_1\,p_\mu&=
\sum_{r, s\ge 1} rsm_r(\mu)(m_s(\mu)-\delta_{rs}) 
\,p_{\mu \setminus (r,s)\cup (r+s-1)}+N(N-1)m_1(\mu)p_{\mu \setminus (1)}\\
&+\sum_{r\ge 3} rm_r(\mu) \sum_{i=1}^{r-2}
\,p_{\mu \setminus (r)\cup (i,r-i-1)}
+(2N-1)\sum_{r\ge 2}rm_r(\mu)\,p_{\mu \setminus (r)\cup (r-1)}.
\end{split}
\end{equation}
We have also
\begin{equation}
\begin{split}
E_0\,p_{\mu}&=\sum_{r\ge 2}rm_{r}(\mu) \,p_{\mu \setminus (r)\cup (r-1)}+Nm_1(\mu)p_{\mu\setminus (1)},\\ E_2\,p_{\mu}&=\sum_{r\ge 1}rm_{r}(\mu) \,p_{\mu \setminus (r)\cup (r+1)}.
\end{split}
\end{equation}
In these expressions, observe that the term $m_1(\mu)p_{\mu \setminus (1)}$ may also been written ${\partial}/{\partial p_1}(p_\mu)$. 

Finally we introduce the differential operators
\begin{equation*}
\begin{split}
\mathbf{E}&=E_0-N\frac{\partial}{\partial p_1},\\ 
\mathbf{D}&=2D_1-(2N-1)\mathbf{E}-N(N-1)\frac{\partial}{\partial p_1}.
\end{split}
\end{equation*} 
These operators are independent of $N$ because using (4.2)-(4.3), their action on power sums is given by 
\begin{equation*}
\begin{split}
\mathbf{E}\,p_{\mu}&=\sum_{r\ge 2}rm_{r}(\mu) \,p_{\mu \setminus (r)\cup (r-1)},\\ 
\mathbf{D}\,p_{\mu}&=\sum_{r, s\ge 1} rsm_r(\mu)(m_s(\mu)-\delta_{rs}) 
\,p_{\mu \setminus (r,s)\cup (r+s-1)}\\
&+\sum_{r\ge 3} rm_r(\mu) \sum_{i=1}^{r-2}
\,p_{\mu \setminus (r)\cup (i,r-i-1)}.
\end{split}
\end{equation*}
If we denote $\widehat{p}_{\mu}=z_{\mu}^{-1}p_{\mu}$ we have
\begin{equation}
\begin{split}
\mathbf{E}\,\widehat{p}_{\mu}=&\sum_{r\ge 1}r(m_{r}(\mu)+1)\, \widehat{p}_{\mu \setminus (r+1)\cup (r)},\\
\mathbf{D}\,\widehat{p}_{\mu}=&
\sum_{r,s\ge 1} (r+s-1)(m_{r+s-1}(\mu)+1)\, \widehat{p}_{\mu \setminus (r,s)\cup (r+s-1)}\\
+&\sum_{r,s \ge 1} rs(m_r(\mu)+1)(m_s(\mu)+1+\delta_{rs})\,
\widehat{p}_{\mu \setminus (r+s+1)\cup (r,s)},
\end{split}
\end{equation}
which is an elementary consequence of 
\begin{align*}
z_{\mu}^{-1}z_{\mu \setminus (r+1)\cup(r)}
&=\frac{r(m_r(\mu)+1)}{(r+1)m_{r+1}(\mu)},\\
z_{\mu}^{-1}z_{\mu \setminus (r,s)\cup (r+s-1)}&=\frac{(r+s-1)(m_{r+s-1}(\mu)+1)}{rsm_r(\mu)(m_s(\mu)-\delta_{rs})},\\
z_{\mu}^{-1}z_{\mu \setminus (r+s+1)\cup (r,s)}&=\frac{rs(m_r(\mu)+1)(m_s(\mu)+1+\delta_{rs})}{(r+s+1)m_{r+s+1}(\mu)}.
\end{align*}

\subsection{Action on Schur functions}

Let $\Delta_0= p_1$ considered as a multiplication operator acting on the symmetric algebra $\mathbb{S}$. For any $k \ge 1$ define the $k$-th nested commutator 
\begin{equation*}
\Delta_k= 
[D_2,[D_2,\cdots,[D_2,p_1]\cdots]].
\end{equation*} 
After some easy but tedious computation we have
\begin{equation}
\begin{split}
\Delta_1&=E_2+(N-1)p_1,\\
\Delta_2&=2D_3+E_2+(N-1)^2p_1.
\end{split}
\end{equation}

The Schur functions $s_{\la}(x_1,\ldots,x_N)$ are eigenfunctions of $D_2$, namely
\begin{equation*}
D_2 s_\la=\big(p_1(A_{\la})+|\la|(N-1)\big)s_\la.
\end{equation*}
A proof is obtained by setting $\alpha=1$ in a more general result, proved in~\cite[Theorem 3.1, p. 84]{S0}. This property and the Pieri formula
\[p_1s_\la=\sum_{i=1}^{l(\la)+1} s_{\la^{(i)}},\]
imply inductively
\begin{equation}
\Delta_k s_\la=\sum_{i=1}^{l(\la)+1} (\la_i+N-i)^k\, s_{\la^{(i)}}.
\end{equation} 
 
As a consequence of~\cite[Example 1.5.5, p. 75]{Ma} we have $p_1^\perp={\partial}/{\partial p_1}$, hence
\[\frac{\partial}{\partial p_1}s_\la= \sum_{i=1}^{l(\la)} s_{\la_{(i)}}.\]
From~\cite[Example 1.3.10, p. 47]{Ma} we deduce
\[
E_0s_\la= \lim\limits_{t\rightarrow 0} \,t^{-1} (s_\la(x_1+t,\ldots,x_N+t)-s_\la(x_1,\ldots,x_N))=\sum_{i=1}^{l(\la)} (N+\la_i-i) s_{\la_{(i)}}.\]
Therefore we obtain 
\[2D_1s_\la=[E_0,D_2]s_\la= \sum_{i=1}^{l(\la)} (N+\la_i-i)(N-1+\la_i-i) s_{\la_{(i)}},\]
Finally we get
\begin{equation}
\begin{split}
\mathbf{E}s_\la&= \sum_{i=1}^{l(\la)} (\la_i-i) s_{\la_{(i)}},\\
\mathbf{D}s_\la&= \sum_{i=1}^{l(\la)} (\la_i-i)^2 s_{\la_{(i)}},
\end{split}
\end{equation}
which provides another proof of the independence on $N$ of these operators.

\subsection{Central characters}

We are now in a position to obtain some linear relations between central characters. Our purpose is to evaluate
\[\sum_{i=1}^{l(\la)+1} 
c_i(\la) \, (\la_i-i+1)^k\,\ta^{\la^{(i)}}_\mu,\]
at least for the first values of $k$.
\begin{theo}
For any partitions $\la \vdash n$ and $\mu \vdash n+1$, we have
\begin{equation}
\sum_{i=1}^{l(\la)+1} 
c_i(\la) \, \ta^{\la^{(i)}}_\mu=\ta^\la_{\mu \setminus (1)},
\end{equation}
\begin{equation}
\sum_{i=1}^{l(\la)+1} 
c_i(\la) \, (\la_i-i+1)\,\ta^{\la^{(i)}}_\mu=\sum _{r\ge 1}
r(m_r(\mu)+1)\,\ta^\la_{\mu \setminus (r+1)\cup (r)},
\end{equation}
\begin{equation}
\begin{split}
\sum_{i=1}^{l(\la)+1} 
c_i(\la) \, (\la_i-i+1)^2\,\ta^{\la^{(i)}}_\mu&=
(2n-m_1(\mu)+1) \, \ta^\la_{\mu \setminus 1}\\
&+\sum_{r,s \ge 1} rs(m_r(\mu)+1)(m_s(\mu)+\delta_{rs}+1)\, 
\ta^{\la}_{\mu \setminus (r+s+1)\cup (r,s)}
\\ &+ \sum_{r,s \ge 2} (r+s-1)(m_{r+s-1}(\mu)+1)\,
\ta^{\la}_{\mu \setminus (r,s)\cup (r+s-1)}.
\end{split}
\end{equation}
\end{theo}
\begin{proof}
Since $\ta^\la_\mu=H_\la z_\mu^{-1}\chi^\la_\mu$, we may write the classical Frobenius formula
\begin{equation*}
s_{\la}=\sum_{\mu} z_\mu^{-1}\chi^\la_\mu \, p_\mu 
\end{equation*}
under the equivalent form
\[H_\la s_{\la}=\sum_{\mu} \ta^\la_\mu \, p_\mu.\]
We apply the differential operator $\Delta_k$ on both sides. As a consequence of (4.6) and $H_\la=c_i(\la) H_{\la^{(i)}}$ we obtain
\begin{align*}
H_\la \,\Delta_k \, s_\la
&=\sum_{i=1}^{l(\la)+1} c_i(\la) \, (\la_i+N-i)^k\, H_{\la^{(i)}}s_{\la^{(i)}}\\
&= 
\sum_{\nu} \Big( \sum_{i=1}^{l(\la)+1} c_i(\la) \, (\la_i+N-i)^k\,\ta^{\la^{(i)}}_\nu \Big) \,p_\nu\\
&=\sum_{\mu} \ta^\la_\mu \, \Delta_k\,p_\mu.
\end{align*}
Let us write this identity for $k=2$. By (4.5), (4.1) and (4.3) we get
\begin{align*}
\sum_{\nu} \Big( \sum_{i=1}^{l(\la)+1} &c_i(\la) \, (\la_i+N-i)^2\,\ta^{\la^{(i)}}_\nu \Big) \,p_\nu \\
&=\sum_{\mu} \ta^\la_\mu \,\Big(
\sum_{r, s\ge 1} rsm_r(\mu)(m_s(\mu)-\delta_{rs}) 
\,p_{\mu \setminus (r,s)\cup (r+s+1)} \\
&+\sum_{r\ge 1} rm_r(\mu) \sum_{i=1}^{r}
\,p_{\mu \setminus (r)\cup (i,r-i+1)}
+(2N-3)\sum_{r\ge 1}rm_r(\mu)\,p_{\mu \setminus (r)\cup (r+1)}\\
&+\sum_{r\ge 1}rm_{r}(\mu) \,p_{\mu \setminus (r)\cup (r+1)}
+(N-1)^2 \,p_{\mu \cup (1)}\Big).
\end{align*}
If we denote by $L^{(k)}_\mu$, $(k=0,1,2)$, the respective left-hand sides of (4.8)--(4.10), we have
\[\sum_{\nu} \Big( \sum_{i=1}^{l(\la)+1} c_i(\la) \, (\la_i+N-i)^2\,\ta^{\la^{(i)}}_\nu \Big) \,p_\nu=\sum_{\nu} \Big(
L^{(2)}_\nu +2(N-1)L^{(1)}_\nu+ (N-1)^2 L^{(0)}_\nu\Big) \,p_\nu.\]
Since these quantities are independent of $N$, we obtain $\sum_{\nu}L^{(k)}_\nu p_\nu$ by identification of the coefficients of $N-1$. 

Finally we identify the coefficients of power sums on both sides. Relations (4.8) and (4.9) are straightforward. For (4.10) some attention is needed with the term
\begin{equation}
\sum_{r\ge 1} rm_r(\mu) \sum_{i=1}^{r}
\,p_{\mu \setminus (r)\cup (i,r-i+1)},
\end{equation}
which should be written as
\[\sum_{r,s \ge 2} (r+s-1)m_{r+s-1}(\mu)\,
p_{\mu \setminus (r+s-1)\cup (r,s)}+(2|\mu| -m_1(\mu)) \,p_{\mu \cup (1)}.\]
The last term $(2|\mu|-m_1(\mu)) p_{\mu \cup (1)}$ is justified as follows. In (4.11) for $r\neq 1$, each case $i=1$ and $i=r$ contributes by $p_{\mu \cup (1)}$. For $r=1$ there is only one such contribution, obtained for $i=r=1$. Hence a total of $m_1(\mu)+2\sum_{r\ge 2} rm_r(\mu)=2|\mu|-m_1(\mu)$.
\end{proof}

\noindent\textit{Remark: }F\'eray~\cite[Appendix]{Fer3} has observed that
\[\sum_{i=1}^{l(\la)+1} 
c_i(\la) \, (\la_i-i+1)^k\,\ta^{\la^{(i)}}_\mu=
\hat{\chi}^\lambda( \pi( J_{n+1}^k C_\mu )),\]
with $\pi$ the orthogonal projection of $\mathbb{C}[S_{n+1}]$ onto $\mathbb{C}[S_n]$ and $C_\mu\in\mathbb{C}[S_{n+1}]$. This formula provides an interesting connection between the analytic and combinatorial points of view. The proof is an extension of the one given by Biane~\cite[Proposition 3.3]{B1} for $\mu=1^{n+1}$.

\subsection{Our method}

Consider a symmetric function $f$ and its central character expansion
\[f(A_\la)=\sum_{|\mu|=n}  a_{\mu}(n) \, \ta^\la_\mu.\]
We sketch the main steps of our method to compute $a_{\mu}(n)$. 

\noindent\textit{First step:} By definition we have $A_{\la^{(i)}}=A_{\la}\cup \{\la_i-i+1\}$. Therefore we may write
\begin{equation}
f(A_{\la^{(i)}})=f(A_\la)+ \sum_{k\ge 1} g_k(A_\la) (\la_i-i+1)^k,
\end{equation}
for a finite family of symmetric functions $g_k$. When $f$ is specified this development may be found explicitly. But the existence of $g_k$ is a general fact~\cite[Example 1.5.3 (b), p. 75]{Ma}: actually $g_k=h_k^{\perp}f$.

\noindent\textit{Second step:} Using (3.1) the previous expansion implies for $r=0,1,2$,
\begin{equation}
\sum_{i=1}^{l(\la)+1} c_i(\la)\,(\la_i-i+1)^r\,f(A_{\la^{(i)}})=
\sigma_r(\la) f(A_\la)+ \sum_{k\ge 1} \sigma_{k+r}(\la)g_k(A_\la).
\end{equation}
Then it may be possible (but not always) to eliminate the quantities $\sigma_i(\la)$. This can be done by performing some linear combinations and using the explicit expressions (3.5)--(3.6). This elimination depends strongly on the specific form of $f$. 

In the most elementary situation (always encountered in this paper), this elimination transforms (4.13) into
\begin{equation}
\sum_{i=1}^{l(\la)+1} c_i(\la)\,(\la_i-i+1)^r\,f(A_{\la^{(i)}})=
F(A_\la)+\sum_{i=1}^{l(\la)+1} c_i(\la)\,(\la_i-i+1)^s\,G(A_{\la^{(i)}}),
\end{equation}
for some $s=0,1,2$ and some symmetric functions $F$, $G$ (all three depending on $r$). 

\noindent\textit{Third step:} Writing the content evaluation of these functions as
\[F(A_\la)=\sum_{|\mu|=n}  F_{\mu}(n) \, \ta^\la_\mu, \quad
G(A_\la)=\sum_{|\mu|=n}  G_{\mu}(n) \, \ta^\la_\mu,\]
the previous relation becomes
\begin{align*}
&\sum_{|\mu|=n+1}  a_{\mu}(n+1) \Big(\sum_{i=1}^{l(\la)+1} c_i(\la)\,(\la_i-i+1)^r\,  \ta^{\la^{(i)}}_\mu\Big)\\=&
\sum_{|\nu|=n}  F_{\nu}(n) \, \ta^\la_\nu +
\sum_{|\mu|=n+1}  G_{\mu}(n+1) \Big(\sum_{i=1}^{l(\la)+1} c_i(\la)\,(\la_i-i+1)^s\,  \ta^{\la^{(i)}}_\mu\Big).
\end{align*}
Applying (4.8)--(4.10), the quantities between brackets can be evaluated in terms of the characters $\ta^{\la}_\nu$. 

\noindent\textit{Fourth step:} The central characters $\ta^{\la}_\nu$'s are linearly independent. By identification of their coefficients, we obtain some linear relations between $a_{\mu}(n+1)$ and the $F_{\mu}(n)$, $G_{\mu}(n+1)$'s. 

\noindent\textit{Final step:} These relations may be used to define $a_{\mu}(n)$ inductively.  

This method will appear much clearer below, when applied to $f=e_k, p_k, h_k$ and the Hall-Littlewood function $f=P_k(z)$.

\section{Elementary functions}

As an easy example, let us first apply our method to recover the classical result of  Jucys~\cite{Ju}.
\begin{theo}
For any positive integer $k$ we have
\[e_k(J_1,\ldots,J_n)=\sum_{\begin{subarray}{c}|\mu|=n\\
l(\mu)=n-k \end{subarray}}  C_\mu.\]
\end{theo}
\begin{proof} Writing
\[e_k(A_\la)=\sum_{|\mu|=n} a^{(k)}_\mu(n) \, \ta_\mu^\la,\]
we must equivalently prove that $a^{(k)}_\mu(n)=\delta_{l(\mu),n-k}$.
 
\noindent\textit{First step:} Denoting $u_i=\la_i-i+1$, the generating function $E_z(A_{\la^{(i)}})$ satisfies
\[E_z(A_{\la^{(i)}})=E_z(A_\la \cup u_i)=E_z(A_{\la})(1+zu_i).\]
Hence the expansion (4.12) takes the very simple form
\begin{equation*}
e_k(A_{\la^{(i)}})=e_k(A_{\la})+e_{k-1}(A_{\la})u_i.
\end{equation*}

\noindent\textit{Second step:} Using (3.1) and (3.5) this yields
\begin{align}
\notag \sum_{i=1}^{l(\la)+1} 
c_i(\la) \,e_k(A_{\la^{(i)}})&=\sum_{i=1}^{l(\la)+1} 
c_i(\la) \, \big(e_k(A_{\la})+e_{k-1}(A_{\la})u_i\big)=e_k(A_{\la}),\\
\sum_{i=1}^{l(\la)+1} 
c_i(\la) \, u_i\,e_k(A_{\la^{(i)}})&=\sum_{i=1}^{l(\la)+1} 
c_i(\la) \, u_i  \big(e_k(A_\la)+e_{k-1}(A_\la)u_i\big)=n e_{k-1}(A_\la),\\
\notag \sum_{i=1}^{l(\la)+1} 
c_i(\la) \, u_i^2\,e_k(A_{\la^{(i)}})&=\sum_{i=1}^{l(\la)+1} 
c_i(\la) \, u_i^2 \big(e_k(A_\la)+e_{k-1}(A_\la)u_i\big)=(n e_k+2e_1e_{k-1})(A_\la).
\end{align}
With the notations of Section 4.4, we have $G=0$ and
$F=e_k$, $n e_{k-1}$, $n e_k+2e_1e_{k-1}$, respectively.

\noindent\textit{Third step:} Applying (4.8) and (4.9) the two first relations may be written as
\begin{equation*}
\begin{split}
\sum_{|\mu|=n+1} a^{(k)}_\mu(n+1) \ta_{\mu \setminus  (1)}^\la&=\sum_{|\nu|=n} a^{(k)}_\nu(n) \ta_\nu^\la,\\
\sum_{|\mu|=n+1} a_\mu^{(k)}(n+1)\, \sum _{r\ge 1}
r\big(m_r(\mu)+1\big)\theta^\la_{\mu \setminus (r+1)\cup (r)}&=
n\sum_{|\nu|=n} a^{(k-1)}_\nu(n) \ta_\nu^\la.
\end{split}
\end{equation*}

\noindent\textit{Fourth step:} By identification of coefficients on both sides, for any $\mu\vdash n$ we obtain
\begin{equation}
\begin{split}
a^{(k)}_{\mu\cup(1)}(n+1)&=a^{(k)}_\mu(n),\\
\sum _{r\ge 1}
rm_r(\mu)\,a^{(k)}_{\mu \setminus (r)\cup (r+1)}(n+1)
&=n a^{(k-1)}_{\mu}(n).
\end{split}
\end{equation}
\noindent\textit{Final step:} The previous recurrence relations allow us to conclude by a triple induction: firstly on $k$, then on $n$ and finally on the lowest part of $\mu\vdash n$. 

We warn the reader that we shall use such a multiple induction several times in this paper. Details are given here, but will not be repeated below.

\textit{(i)} Assume $a^{(i)}_\mu(n)=\delta_{l(\mu),n-i}$ for any $n$ and $\mu$, and $i\le k-1$.

\textit{(ii)} Assume $a^{(k)}_\mu(m-1)=\delta_{l(\mu),m-k-1}$ for any $\mu \vdash m-1$. By the first relation (5.2) we have obviously $a^{(k)}_\mu(m)=\delta_{l(\mu),m-k}$ for those $\mu \vdash m$ whose lowest part is $1$. 

\textit{(iii)} Now assume this property to be true for those $\mu \vdash m$ whose lowest part is $p-1$. Let $\nu \vdash m$ having lowest part $p$. Then $\mu=\nu \setminus (p)\cup (p-1)\vdash m-1$ has lowest part $p-1$ with multiplicity 1. 

Writing the second relation (5.2) for $\mu$ determines $a^{(k)}_{\nu}(m)$, since all partitions on the left-hand side have lowest part $p-1$, but $\nu=\mu \setminus (p-1)\cup (p)$. Namely
\begin{equation*}
\begin{split}
(p-1)\,a^{(k)}_{\nu}(m)&=(m-1) a^{(k-1)}_{\mu}(m-1)-\sum _{r> p-1}
rm_r(\mu)\,a^{(k)}_{\mu \setminus (r)\cup (r+1)}(m)\\
&=\left(m-1-\sum _{r> p-1}
rm_r(\mu)\right)\delta_{l(\mu),m-k}\\
&=(p-1)\delta_{l(\nu),m-k}.
\end{split}
\end{equation*}
\end{proof}

Up to now we have only used the two first equations (5.1). But the third one has also an interesting consequence.
\begin{prop}
For any positive integer $k$ we have
\[(e_1e_k)(J_1,\ldots,J_n)=\sum_{\begin{subarray}{c}|\mu|=n\\
l(\mu)=n-k-1 \end{subarray}}  a_\mu C_\mu
+\sum_{\begin{subarray}{c}|\mu|=n\\
l(\mu)=n-k+1 \end{subarray}} \left(\binom{n}{2}-a_\mu\right) C_\mu,\]
with $a_\mu=\sum_{r\ge 2}m_r(\mu) \binom{r}{2}$. 
\end{prop}
\begin{proof}
The third relation (5.1) may be written as
\[2(e_ke_1)(A_{\la}):=\sum_{|\mu|=n}  2a^{(k,1)}_\mu(n)\ta^\la_\mu=-n e_{k+1}(A_{\la})+\sum_{i=1}^{l(\la)+1} 
c_i(\la) \,(\la_i-i+1)^2\,e_{k+1}(A_{\la^{(i)}}).\]
Using (4.10), the expansion of the right-hand side is
\begin{align*}
-n\sum_{\begin{subarray}{c}|\mu|=n\\
l(\mu)=n-k-1 \end{subarray}}  \ta^\la_\mu
+\sum_{\begin{subarray}{c}|\mu|=n+1\\
l(\mu)=n-k \end{subarray}}
&\Big((2n-m_1(\mu)+1) \, \ta^\la_{\mu \setminus 1}\\
&+\sum_{r,s \ge 1} rs(m_r(\mu)+1)(m_s(\mu)+\delta_{rs}+1)\, 
\ta^{\la}_{\mu \setminus (r+s+1)\cup (r,s)}
\\ &+ \sum_{r,s \ge 2} (r+s-1)(m_{r+s-1}(\mu)+1)\,
\ta^{\la}_{\mu \setminus (r,s)\cup (r+s-1)}\Big).
\end{align*}
By identification of coefficients on both sides, we obtain
\begin{equation*}
\begin{split}
2a^{(k,1)}_\mu(n)=-n\delta_{l(\mu),n-k-1}&+\sum_{r,s \ge 1} rsm_r(\mu)(m_s(\mu)-\delta_{rs})\, 
\delta_{l(\mu)-1,n-k}
\\ &+\sum_{r,s \ge 1} (r+s-1)m_{r+s-1}(\mu)\,\delta_{l(\mu)+1,n-k}.
\end{split}
\end{equation*}
The coefficient of $\delta_{l(\mu),n-k-1}$ is 
\begin{equation}
-n+\sum_{r,s \ge 1} (r+s-1)m_{r+s-1}(\mu)=-n+\sum_{t\ge 1} t^2m_t(\mu)=2a_\mu.
\end{equation}
The coefficient of $\delta_{l(\mu),n-k+1}$ is 
\begin{equation}
\sum_{r,s \ge 1} rsm_r(\mu)(m_s(\mu)-\delta_{rs})=n^2-\sum_{t\ge 1} t^2m_t(\mu)=n^2-n-2a_\mu.
\end{equation}
\end{proof}

\section{Power sums}

We are now in a position to give an alternative proof of the classical result of Lascoux-Thibon~\cite{LT}. We study the expansion
\[p_k(A_\la)=\sum_{|\mu|=n} a^{(k)}_\mu(n) \, \ta_\mu^\la\]
in two steps: firstly a recurrence between coefficients, secondly a generating function. 

\subsection{Recurrence}

Denoting $u_i=\la_i-i+1$, the expansion (4.12) takes the very simple form
\begin{equation*}
p_k(A_{\la^{(i)}})=p_k(A_{\la})+u_i^k.
\end{equation*}
Therefore (3.1) and (3.5) yield
\begin{align*}
\sum_{i=1}^{l(\la)+1} 
c_i(\la) \,p_k(A_{\la^{(i)}})&=p_k(A_{\la})+\sigma_k(\la),\\
\sum_{i=1}^{l(\la)+1} 
c_i(\la) \,u_i\,p_k(A_{\la^{(i)}})&=\sigma_{k+1}(\la),\\
\sum_{i=1}^{l(\la)+1} 
c_i(\la) \,u_i^2\,p_k(A_{\la^{(i)}})&
=np_k(A_{\la})+\sigma_{k+2}(\la).
\end{align*}
By elimination we get immediately
\begin{align*}
\sum_{i=1}^{l(\la)+1} 
c_i(\la) \,p_{k}(A_{\la^{(i)}})&=p_{k}(A_{\la})+\sum_{i=1}^{l(\la)+1} 
c_i(\la) \,u_i\,p_{k-1}(A_{\la^{(i)}}),\\
\sum_{i=1}^{l(\la)+1} 
c_i(\la) \,u_i\,p_{k}(A_{\la^{(i)}})&=-np_{k-1}(A_\la)
+\sum_{i=1}^{l(\la)+1} 
c_i(\la) \,u_i^2\,p_{k-1}(A_{\la^{(i)}}).
\end{align*}
Applying (4.8)--(4.10) these relations write respectively as
\begin{align*}
\sum_{|\mu|=n+1} a^{(k)}_\mu(n+1)\, \ta_{\mu \setminus  (1)}^\la&=\sum_{|\nu|=n} a^{(k)}_\nu(n)\, \ta_\nu^\la\\
&+\sum_{|\mu|=n+1} a_\mu^{(k-1)}(n+1)\, \sum _{r\ge 1}
r\big(m_r(\mu)+1\big)\theta^\la_{\mu \setminus (r+1)\cup (r)},
\end{align*}
\begin{align*}
\sum_{|\mu|=n+1} a_\mu^{(k)}(n+1)\, &\sum _{r\ge 1}
r\big(m_r(\mu)+1\big)\theta^\la_{\mu \setminus (r+1)\cup (r)}=
-n\sum_{|\nu|=n} a^{(k-1)}_{\nu}(n)\, \ta_\nu^\la\\
&+\sum_{|\mu|=n+1} a_\mu^{(k-1)}(n+1)\,
\Big((2n-m_1(\mu)+1) \, \ta^\la_{\mu \setminus 1}\\
&+\sum_{r,s \ge 1} rs(m_r(\mu)+1)(m_s(\mu)+\delta_{rs}+1)\, 
\ta^{\la}_{\mu \setminus (r+s+1)\cup (r,s)}
\\ &+ \sum_{r,s \ge 2} (r+s-1)(m_{r+s-1}(\mu)+1)\,
\ta^{\la}_{\mu \setminus (r,s)\cup (r+s-1)}\Big).
\end{align*}
By identification of coefficients on both sides, for any $\mu\vdash n$ we get
\begin{equation}
a^{(k)}_{\mu\cup(1)}(n+1)=a^{(k)}_\mu(n)+\sum _{r\ge 1}
rm_r(\mu)\,a^{(k-1)}_{\mu \setminus (r)\cup (r+1)}(n+1),
\end{equation}
\begin{equation}
\begin{split}
\sum _{r\ge 1}
rm_r(\mu)\,a^{(k)}_{\mu \setminus (r)\cup (r+1)}(n+1)&=
-na^{(k-1)}_{\mu}(n)\\
&+\sum_{r,s \ge 1} rsm_r(\mu)(m_s(\mu)-\delta_{rs})\, 
a_{\mu \setminus (r,s)\cup (r+s+1)}^{(k-1)}(n+1)
\\ &+ \sum_{r,s \ge 1} (r+s-1)m_{r+s-1}(\mu)\,
a_{\mu \setminus (r+s-1)\cup (r,s)}^{(k-1)}(n+1).
\end{split}
\end{equation}
In the second sum on the right-hand side, the cases $r=1$ or $s=1$ give the total contribution $(2n-m_1(\mu))\,a_{\mu \cup (1)}^{(k-1)}(n+1)$.

These two recurrence relations determine the coefficients $a^{(k)}_\mu(n)$ by a triple induction: on $k$, on $n$ and on the lowest part of $\mu$. Firstly (6.1) gives $a^{(k)}_\mu(n+1)$ for those $\mu$ whose lowest part is $1$. Then (6.2) determines $a^{(k)}_\mu(n+1)$ by induction on the lowest part of $\mu$. The proof is strictly parallel to the final step of Theorem 5.1 and is left to the reader.

These recurrence relations also imply that $a^{(k)}_{\mu}(n)$ may be written as
\begin{equation*}
a^{(k)}_{\mu}(n)=\sum_{\overline{\rho}=\overline{\mu}}c^{(k)}_\rho \, \binom{n-|\overline{\rho}|}{m_1(\rho)},
\end{equation*}
with $c^{(k)}_\rho$ independent of $n$ (a property already known in view of Section 2.8). By substitution into (6.1)--(6.2), we obtain the corresponding recurrence relations for $c^{(k)}_\rho$.

This computation will be used several times in this paper. Since it is rather technical, we postpone it to an Appendix. Using the Lemma given there with $z=1$, we obtain the following result.
\begin{theo}
We have the class expansion
\[p_k(J_1,\ldots,J_n)=\sum_{\rho}c^{(k)}_\rho \, \binom{n-|\overline{\rho}|}{m_1(\rho)} \, C_{\tilde{\rho}},\]
where the coefficients $c^{(k)}_\rho$ are determined by the two recurrence relations
\begin{equation}
c^{(k)}_{\rho\cup (1)}=\sum_{r \ge 1} rm_r(\rho) \,c^{(k-1)}_{\rho\setminus (r) \cup(r+1)},
\end{equation}
\begin{equation}
\begin{split}
\sum_{r \ge 1} rm_r(\rho) \,c^{(k)}_{\rho\setminus (r) \cup(r+1)}=
|\rho|\,c^{(k-1)}_\rho
&+\sum_{r,s \ge 1} rsm_r(\rho)(m_s(\rho)-\delta_{rs})\, 
c_{\rho \setminus (r,s)\cup (r+s+1)}^{(k-1)}\\
&+\sum_{r,s \ge 1} (r+s-1)m_{r+s-1}(\rho)\,
c_{\rho \setminus (r+s-1)\cup (r,s)}^{(k-1)}.
\end{split}
\end{equation}
\end{theo}

\noindent\textit{Remarks: (i)} In the second sum on the right-hand side of (6.4), the cases $r=1$ or $s=1$ give the total contribution $(2|\rho|-m_1(\rho))\,c_{\rho \cup (1)}^{(k-1)}$.

\noindent\textit{(ii)} As above, $c_{\rho}^{(k)}$ is inductively defined by (6.3) for $\rho$ having lowest part 1, and by (6.4) for $\rho$ having lowest part $>1$.

\noindent\textit{(iii)} By induction on $k$ and the lowest part of $\rho$, the coefficients $c_{\rho}^{(k)}$ are non zero for $|\rho|+l(\rho)= k+2-2i$ for some $i\ge 0$.

\subsection{Generating function}

We look for a generating function of the coefficients $c_{\rho}^{(k)}$, written under the form
\[\phi_\rho(t)=\sum_{k\ge 0} c_{\rho}^{(k)} \, \frac{t^k}{k!}.\]
It will be useful to collect the $\phi_\rho$'s for partitions $\rho$ having the same weight $w$. For that purpose, we introduce a set of $N$ auxiliary variables $X=(x_1,\ldots,x_N)$, and we define
\[\Phi_w(t;X)= \sum_{|\rho|=w}\phi_\rho(t)\, z_{\rho}^{-1}p_\rho(X).\]

We shall translate the recurrence relations (6.3)--(6.4) in terms of $\Phi_w$. This will allow us to evaluate $\Phi_w$, hence $\phi_\rho$.

\begin{prop} 
The recurrence relations (6.3)--(6.4) are equivalent with
\begin{equation*}
\frac{d}{dt}\,\frac{\partial}{\partial p_1}\Phi_w=\mathbf{E}\, \Phi_w,
\qquad \frac{d}{dt}\,\mathbf{E}\,\Phi_{w+1}= \mathbf{D}\,\Phi_{w+1}+w\Phi_w.
\end{equation*}
\end{prop}
\begin{proof}
Denoting $\widehat{p}_{\mu}=z_{\mu}^{-1}p_{\mu}$, (4.4) implies
\begin{align*}
\frac{d}{dt}\,\frac{\partial}{\partial p_1}\Phi_w&=
\sum_{|\rho|=w} \sum_{k\ge 0} c_{\rho}^{(k)} \, \frac{t^{k-1}}{(k-1)!}\,\widehat{p}_{\rho\setminus (1)},\\
\mathbf{E}\,\Phi_w&=\sum_{|\rho|=w}\sum_{k\ge 0} c_{\rho}^{(k)} \, \frac{t^{k}}{k!} \,\sum_{r\ge 1}r(m_{r}(\rho)+1)\, \widehat{p}_{\rho \setminus (r+1)\cup (r)}\\
\mathbf{D} \,\Phi_{w+1}&=\sum_{|\rho|=w+1}\sum_{k\ge 0} c_{\rho}^{(k)} \, \frac{t^{k}}{k!} \Big(\sum_{r,s\ge 1} (r+s-1)(m_{r+s-1}(\rho)+1)\, \widehat{p}_{\rho \setminus (r,s)\cup (r+s-1)}\\
&+\sum_{r,s \ge 1} rs(m_r(\rho)+1)(m_s(\rho)+1+\delta_{rs})\,
\widehat{p}_{\rho \setminus (r+s+1)\cup (r,s)}\Big).
\end{align*}
We conclude by identification of coefficients of $t^{k-1}\widehat{p}_\rho$.
\end{proof}

Instead of power sums, we may alternatively decompose $\Phi_w(t;X)$ in terms of Schur functions, and write
\[\Phi_w(t;X)= \sum_{|\rho|=w}\psi_\rho(t)\, s_\rho(X).\]
Using (4.7), by identification of coefficients of Schur functions, Proposition 6.2 writes equivalently as
\begin{equation}
\begin{split}
\sum_{i=1}^{l(\rho)+1}\frac{d}{dt} \,\psi_{\rho^{(i)}}(t)&=
\sum_{i=1}^{l(\rho)+1}(\rho_i-i+1) \,\psi_{\rho^{(i)}}(t),\\
\sum_{i=1}^{l(\rho)+1}(\rho_i-i+1)\frac{d}{dt} \,\psi_{\rho^{(i)}}(t)&=
\sum_{i=1}^{l(\rho)+1}(\rho_i-i+1)^2 \,\psi_{\rho^{(i)}}(t)+|\rho|\psi_\rho(t).
\end{split}
\end{equation}

Because $p_0(J_1,\ldots,J_n)=n$, we have $c^{(0)}_\rho=\delta_{\rho,(1)}$ for any partition $\rho$. Therefore this first order differential system must be solved with the initial conditions $\psi_{\rho}(0)=\delta_{\rho,(1)}$ for any $\rho$.

\begin{prop}
The solutions of the system (6.5) are given by 
\begin{align*}
|\rho| !\,\psi_{\rho}(t)=\begin{cases}
(e^t-1)^{r-1}\, (e^{-t}-1)^s \qquad &\textrm{if} \: \rho \: \mbox{is a hook} \: (r,1^s)\: \textrm{with} \: r\ge 1,\\ 0 &\textrm{otherwise}.\end{cases}
\end{align*}
\end{prop}
\begin{proof}
Let $P(w)$ denote the number of partitions with weight $w$. We have $P(w+1)<2P(w)$ as soon as $w\ge2$. The differential system (6.5) written for $|\rho|=w$ is formed of $2P(w)$ equations and involves $P(w+1)$ indeterminates. Therefore it is overdetermined and has at most one solution. It is enough to check that the given values are solutions. This easy computation is left to the reader.
\end{proof}

We are now in a position to give another proof of the Lascoux-Thibon's result~\cite{LT}. We shall make a limited, but crucial, use of $\la$-ring calculus. Here we shall not enter into details, and refer the reader to~\cite[Chapter 2]{Las} or~\cite[Section 3]{La3} for a short survey of this theory.

If $f$ is a symmetric function, we denote by $f[P]$ its $\lambda$-ring action on any polynomial $P$. Let $q$ be some indeterminate, $X=(x_1,\ldots,x_N)$ an alphabet and $X^\dag=\sum_{i=1}^N x_i$. We have the fundamental Cauchy formulas~\cite[p.13]{Las}
\begin{align*}
h_w [(q-1)X^\dag]&= \sum_{|\rho|=w} 
p_{\rho}[q-1] \, z_{\rho}^{-1}p_{\rho}(X)\\
&=\sum_{|\rho|= w} s_{\rho} [q-1] \, s_{\rho} (X),
\end{align*} 
with~\cite[p.11]{Las}
\begin{align*}
p_{\rho}[q-1]&=\prod_{i\ge 1}(q^{i}-1)^{m_i(\rho)}, \\
s_{\rho}[q-1]&=\begin{cases}
(-1)^s q^{r-1}\, (q-1) \qquad \textrm{if} \: \rho \: \mbox{is a hook} \: (r,1^s)\: \textrm{with} \: r\ge 1,\\ 0 \hspace{3.7cm} \textrm{otherwise}.\end{cases}
\end{align*}
\begin{theo}
In the class expansion
\[p_k(J_1,\ldots,J_n)=\sum_{\rho}c^{(k)}_\rho \, \binom{n-|\overline{\rho}|}{m_1(\rho)} \, C_{\tilde{\rho}},\]
the coefficients $c^{(k)}_\rho$ have the generating function
\begin{equation*}
\sum_{k\ge 0} c_{\rho}^{(k)} \, \frac{t^k}{k!}=
\frac{e^{-t}}{|\rho|!}(1-e^{-t})^{|\rho|-2}\,\prod_{i\ge 1}(e^{it}-1)^{m_i(\rho)}.
\end{equation*}
\end{theo}
\begin{proof}The assertion may be written as
\[\phi_\rho(t)=\frac{e^{-t}}{|\rho|!}(1-e^{-t})^{|\rho|-2}\,
p_\rho[q-1]\arrowvert_{q=e^t}.\]
In view of the first Cauchy formula this is equivalent with
\[\Phi_w(t;X)=\frac{e^{-t}}{w!}(1-e^{-t})^{w-2}\,h_w[(q-1)X^\dag]\arrowvert_{q=e^t}.\]
On the other hand by Proposition 6.3 we have 
\[\psi_\rho(t)=\frac{e^{-t}}{|\rho|!}(1-e^{-t})^{|\rho|-2}\,s_\rho[q-1]\arrowvert_{q=e^t}.\]
The second Cauchy formula allows us to conclude.
\end{proof}

\subsection{Moments}

As a quick by-product of our proof we obtain the explicit form of
\[\sigma_k(\la)=\sum_{|\mu|=n} s^{(k)}_\mu(n) \, \theta^\la_{\mu}= \sum_{\rho}\mathbf{s}^{(k)}_\rho \, \binom{n-|\overline{\rho}|}{m_1(\rho)} \, \theta^\la_{\tilde{\rho}},\]
which gives also the class expansion of the central element $M^{(k)}_n$. A very different proof was given independently by F\'eray~\cite{Fer}.

\begin{theo}
The class expansions
\[p_k(J_1,\ldots,J_n)=\sum_{\rho}c^{(k)}_\rho \, \binom{n-|\overline{\rho}|}{m_1(\rho)} \, C_{\tilde{\rho}},\qquad
M^{(k)}_n=\sum_{\rho}\mathbf{s}^{(k)}_\rho \, \binom{n-|\overline{\rho}|}{m_1(\rho)} \, C_{\tilde{\rho}}
\]
are connected by
\[\mathbf{s}^{(k)}_\rho=c^{(k)}_{\rho\cup(1)}.\]
\end{theo}
\begin{proof} At the beginning of Section 6.1, we have seen that
\[\sum_{i=1}^{l(\la)+1} 
c_i(\la) \,p_k(A_{\la^{(i)}})=p_k(A_{\la})+\sigma_k(\la).\]
By (4.8) it yields
\[\sum_{|\mu|=n+1} a^{(k)}_\mu(n+1) \ta_{\mu \setminus  (1)}^\la=\sum_{|\nu|=n} (a^{(k)}_\nu(n)+s^{(k)}_\nu(n))\, \ta_\nu^\la.\]
Equivalently
\begin{align*}
s^{(k)}_\mu(n)&=a^{(k)}_{\mu\cup(1)}(n+1)-a^{(k)}_\mu(n)\\
&=\sum_{\overline{\rho}=\overline{\mu}}c^{(k)}_\rho \left( \binom{n+1-|\overline{\rho}|}{m_1(\rho)}-\binom{n-|\overline{\rho}|}{m_1(\rho)}\right)\\
&=\sum_{\overline{\rho}=\overline{\mu}}c^{(k)}_\rho \, \binom{n-|\overline{\rho}|}{m_1(\rho)-1}.
\end{align*}
\end{proof}

As indicated in Section 3, this result gives the class expansion of $f_k(J_1,\ldots,J_n)$ with $f_k$ a symmetric function depending on $n$, given by (3.4). 

\section{Complete functions}

We now consider the expansion
\[h_k(A_\la)=\sum_{|\mu|=n} a^{(k)}_\mu(n) \, \ta_\mu^\la
=\sum_{\rho}c^{(k)}_\rho \, \binom{n-|\overline{\rho}|}{m_1(\rho)} \,\ta^\la_{\tilde{\rho}}.\]
The proof is very similar, though the situation is much more complicated.

\subsection{Recurrence}

Denoting $u_i=\la_i-i+1$, we have by definition
\[H_z(A_{\la^{(i)}})=H_z(A_{\la})(1-zu_i)^{-1}.\]
Hence the expansion (4.12) may be written as
\begin{equation*}
h_k(A_{\la^{(i)}})=\sum_{j=0}^k h_{k-j}(A_{\la})u_i^j=h_k(A_{\la})+\sum_{j=1}^k h_{k-j}(A_{\la})u_i^j.
\end{equation*}
Using (3.1) and (3.5) we obtain
\begin{align*}
\sum_{i=1}^{l(\la)+1} 
c_i(\la) \,h_k(A_{\la^{(i)}})&=\sum_{i=1}^{l(\la)+1} 
c_i(\la) \, \Big(h_k(A_{\la})+\sum_{j=1}^k h_{k-j}(A_{\la})u_i^j\Big)
\\&=h_k(A_{\la})+\sum_{j=2}^k h_{k-j}(A_{\la})\,\sigma_j(\la).
\end{align*}
Similarly we have
\begin{align*}
\sum_{i=1}^{l(\la)+1} 
c_i(\la) \,u_i\,h_k(A_{\la^{(i)}})&=\sum_{i=1}^{l(\la)+1} 
c_i(\la) \, u_i \Big(h_k(A_{\la})+\sum_{j=1}^k h_{k-j}(A_{\la})u_i^j\Big)
\\&=\sum_{j=1}^k h_{k-j}(A_{\la})\,\sigma_{j+1}(\la)\\
&=\sum_{j=2}^{k+1} h_{k-j+1}(A_{\la})\,\sigma_{j}(\la).
\end{align*}
And also
\begin{align*}
\sum_{i=1}^{l(\la)+1} 
c_i(\la) \,u_i^2\,h_k(A_{\la^{(i)}})&=\sum_{i=1}^{l(\la)+1} 
c_i(\la) \, u_i^2 \Big(\sum_{j=0}^k h_{k-j}(A_{\la})u_i^j\Big)\\&=\sum_{j=0}^k h_{k-j}(A_{\la})\,\sigma_{j+2}(\la)\\
&=\sum_{j=2}^{k+2} h_{k-j+2}(A_{\la})\,\sigma_{j}(\la).
\end{align*}
By elimination, we get immediately
\begin{align*}
\sum_{i=1}^{l(\la)+1} 
c_i(\la) \,h_{k}(A_{\la^{(i)}})&=h_{k}(A_{\la})+\sum_{i=1}^{l(\la)+1} 
c_i(\la) \,u_i\,h_{k-1}(A_{\la^{(i)}}),\\
\sum_{i=1}^{l(\la)+1} 
c_i(\la) \,u_i\,h_{k}(A_{\la^{(i)}})&=\sum_{i=1}^{l(\la)+1} 
c_i(\la) \,u_i^2\,h_{k-1}(A_{\la^{(i)}}).
\end{align*}
Applying (4.8)--(4.10) we obtain
\begin{align*}
\sum_{|\mu|=n+1} a^{(k)}_\mu(n+1) \ta_{\mu \setminus  (1)}^\la&=
\sum_{|\nu|=n} a^{(k)}_\nu(n)\, \ta_\nu^\la\\ &+
\sum_{|\mu|=n+1} a_\mu^{(k-1)}(n+1)\, \sum _{r\ge 1}
r\big(m_r(\mu)+1\big)\theta^\la_{\mu \setminus (r+1)\cup (r)},\\
\sum_{|\mu|=n+1} a_\mu^{(k)}(n+1)\, \sum _{r\ge 1}
r\big(&m_r(\mu)+1\big)\theta^\la_{\mu \setminus (r+1)\cup (r)}
=\\
&\sum_{|\mu|=n+1} a_\mu^{(k-1)}(n+1)\,
\Big((2n-m_1(\mu)+1) \, \ta^\la_{\mu \setminus 1}\\
&+\sum_{r,s \ge 1} rs(m_r(\mu)+1)(m_s(\mu)+\delta_{rs}+1)\, 
\ta^{\la}_{\mu \setminus (r+s+1)\cup (r,s)}
\\ &+ \sum_{r,s \ge 2} (r+s-1)(m_{r+s-1}(\mu)+1)\,
\ta^{\la}_{\mu \setminus (r,s)\cup (r+s-1)}\Big).
\end{align*}
By identification of coefficients on both sides, for any $\mu \vdash n$ we get
\begin{equation}
a^{(k)}_{\mu\cup(1)}(n+1)=a^{(k)}_\mu(n)
+\sum _{r\ge 1}rm_r(\mu)\,a^{(k-1)}_{\mu \setminus (r)\cup (r+1)}(n+1),
\end{equation}
\begin{equation}
\begin{split}
\sum_{r\ge 1}rm_r(\mu)\,a^{(k)}_{\mu \setminus (r)\cup (r+1)}(n&+1)
=\\
&\sum_{r,s \ge 1} rsm_r(\mu)(m_s(\mu)-\delta_{rs})\, 
a_{\mu \setminus (r,s)\cup (r+s+1)}^{(k-1)}(n+1)
\\ +& \sum_{r,s \ge 1} (r+s-1)m_{r+s-1}(\mu)\,
a_{\mu \setminus (r+s-1)\cup (r,s)}^{(k-1)}(n+1).
\end{split}
\end{equation}

As for power sums, these two recurrence relations determine the coefficients $a^{(k)}_\mu(n)$ by a triple induction: on $k$, on $n$ and on the lowest part of $\mu$. Using the Lemma given in the Appendix with $z=0$, we also obtain the following result.
\begin{theo}
In the class expansion
\[h_k(J_1,\ldots,J_n)=\sum_{\rho}c^{(k)}_\rho \, \binom{n-|\overline{\rho}|}{m_1(\rho)} \, C_{\tilde{\rho}},\]
the coefficients $c^{(k)}_\rho$ are determined by the two recurrence relations
\begin{equation}
c^{(k)}_{\rho\cup (1)}=\sum_{r \ge 1} rm_r(\rho) \,c^{(k-1)}_{\rho\setminus (r) \cup(r+1)},
\end{equation}
\begin{equation}
\begin{split}
\sum_{r \ge 1} rm_r(\rho) \,c^{(k)}_{\rho\setminus (r) \cup(r+1)}&=
2|\rho|\,c^{(k-1)}_\rho +m_1(\rho)\,c^{(k-1)}_{\rho\setminus (1)}\\
&+\sum_{r,s \ge 1} rsm_r(\rho)(m_s(\rho)-\delta_{rs})\, 
c_{\rho \setminus (r,s)\cup (r+s+1)}^{(k-1)}\\
&+\sum_{r,s \ge 1} (r+s-1)m_{r+s-1}(\rho)\,
c_{\rho \setminus (r+s-1)\cup (r,s)}^{(k-1)}.
\end{split}
\end{equation}
\end{theo}

\noindent\textit{Remarks:} \textit{(i)} Surprisingly relation (7.1) (resp.(7.3)) is identical with (6.1) (resp. (6.3)).

\noindent\textit{(ii)} In the second sum at the right-hand side of (7.4), the cases $r=1$ or $s=1$ give the total contribution $(2|\rho|-m_1(\rho))\,c_{\rho \cup (1)}^{(k-1)}$.

There is empirical evidence that the coefficients $c_{\rho}^{(k)}$ are \textit{positive integers}. However our recurrence relations are not sufficient to prove this property. A more thorough analysis is needed, which can be found in~\cite[Section 2.4]{Fer2}.

The generating function of the coefficients $c_{\rho}^{(k)}$ will be studied in Section 9.6, in a more general context.

\subsection{Leading terms}

\begin{prop}
The coefficients $c_{\rho}^{(k)}$ are non zero only if $|\rho|-l(\rho)= k-2i$ for some $i\ge 0$. Moreover if $|\rho|-l(\rho)=k$, they are non zero only if $m_1(\rho)=0$.
\end{prop}
\begin{proof}
This is shown by induction on $k$ and on the lowest part $p$ of $\rho$. Firstly if $p=1$, we write (7.3) for $\sigma=\rho \setminus (1)$. At the left-hand side we obtain $c^{(k)}_{\rho}$. By the inductive hypothesis, the right-hand side is non zero only if $|\sigma|+1-l(\sigma)=k-1-2i$ for some $i\ge 0$. Since $|\sigma|-l(\sigma)=|\rho|-l(\rho)$, we conclude that $c^{(k)}_{\rho}$ is non zero only if $|\rho|-l(\rho)=k-2i-2$ for some $i\ge 0$.

In particular if $|\rho|-l(\rho)=k$ and $m_1(\rho)\ge 1$, i.e. if $p=1$, we see that $c^{(k)}_{\rho}$ is necessarily zero.

Secondly if the lowest part of $\rho$ is $p>1$, we write (7.4) for $\sigma=\rho \setminus (p) \cup (p-1)$. We consider induction on $k$ for each contribution at the right-hand side. The terms $c^{(k-1)}_\sigma$ and $c^{(k-1)}_{\sigma\setminus (1)}$ are non zero only if $|\sigma|-l(\sigma)=k-1-2i$ with $i\ge 0$. Since $|\sigma|-l(\sigma)=|\rho|-l(\rho)-1$, these terms are non zero only if $|\rho|-l(\rho)=k-2i$ for some $i\ge 0$. Similarly the first sum is non zero only if $|\sigma|+1-l(\sigma)+1=k-1-2i$ with $i\ge 0$, hence only if $|\rho|-l(\rho)=k-2i-2$ for some $i\ge 0$. Finally the second sum at the right-hand side is non zero only if $|\sigma|+1-l(\sigma)-1=k-1-2i$ with $i\ge 0$, hence only if $|\rho|-l(\rho)=k-2i$ for some $i\ge 0$.

At the left-hand side of (7.4) we obtain
\[
(p-1)c_{\rho}^{(k)}+\sum_{r \ge p} r(m_r(\rho)-\delta_{rp}) \,c^{(k)}_{\rho\setminus (p,r) \cup(p-1,r+1)}.\]
Since any partition $\sigma$ appearing in the sum has lowest part $p-1$ and since $|\sigma|-l(\sigma)=|\rho|-l(\rho)$, we conclude by induction on $p$.
\end{proof}

The following result was proved independently by Murray~\cite{Mur} and Novak~\cite{No1}.
\begin{prop}
If $|\rho|-l(\rho)=k$ and $m_1(\rho)=0$, we have
\[c_{\rho}^{(k)}=\prod_{i=1}^{l(\rho)}C_{\rho_i-1},\]
with $C_r$ the Catalan number $(2r)!/(r+1)!r!$. \end{prop}
\begin{proof}
Let $p$ be the lowest part of $\rho$. Since $m_1(\rho)=0$ we have $p\ge 2$. In the proof above, we have seen that equation (7.4) written for $\sigma=\rho\setminus (p)\cup (p-1)$ defines $c_{\rho}^{(k)}$ inductively. First we assume $p\ge 3$. This equation takes the form
\begin{equation}
\sum_{r \ge 2} rm_r(\sigma) \,c^{(k)}_{\sigma\setminus (r) \cup(r+1)}=
2|\sigma|\,c^{(k-1)}_\sigma +\sum_{r,s \ge 2} (r+s-1)m_{r+s-1}(\sigma)\,
c_{\sigma \setminus (r+s-1)\cup (r,s)}^{(k-1)}.
\end{equation}
Actually at its right-hand side we have $m_1(\sigma)=0$ and the first sum vanishes because it is non zero only if $|\rho|-l(\rho)=k-2i-2$ for some $i\ge 0$, in contradiction with $|\rho|-l(\rho)=k$. Observe that for the same reason, the terms in the second sum with $r=1$ or $s=1$ involve partitions with lowest part 1, and thus vanish. Finally it is enough to substitute the statement into (7.5) and to prove
\[\sum_{r \ge 2} rm_r(\sigma)\frac{C_r}{C_{r-1}}=2\sum_{r \ge 2} rm_r(\sigma)+\sum_{t \ge 2} tm_{t}(\sigma)\,\sum_{\begin{subarray}{c}r,s \ge 2\\r+s=t+1\end{subarray}}
\frac{C_{r-1}C_{s-1}}{C_{t-1}}.\]
But this is an obvious consequence of 
\[C_r=2C_{r-1}+\sum_{i=1}^{r-2}C_iC_{r-i-1},\]
a well known recurrence for Catalan numbers. It remains to consider the case $p=2$, i.e. $\sigma=\rho\setminus (2)\cup (1)$ and $m_1(\sigma)=1$. In this case (7.4) becomes
\[c_{\sigma\setminus (1)\cup (2)}^{(k)}=c_{\sigma\setminus (1)}^{(k-1)},\quad \textrm{i.e.} \quad c_{\rho}^{(k)}=c_{\rho\setminus (2)}^{(k-1)},\]
because all other terms involve partitions with lowest part 1, and thus vanish due to $|\rho|-l(\rho)=k$. The statement is then a consequence of $C_1=1$.
\end{proof}

\section{Hall-Littlewood functions}

Let $z$ be an indeterminate and $P_\la(z)$ denote the Hall-Littlewood symmetric functions~\cite[Chapter 3]{Ma}. When $\la$ is the row partition $(k)$, it is known that $P_k(z)$ interpolates between the power-sum $p_k$ and the complete function $h_k$, namely
\[P_k(0)=h_k,\qquad P_k(1)=p_k.\]
Therefore it is a natural problem to study the development of $P_k(J_1,\ldots,J_n;z)$. This study will provide a deeper understanding of the case of complete functions through a continuous deformation of the case of power sums.

For clarity of display, the parameter $z$ being kept fixed, we shall omit its dependence any time it does not bring confusion. However we emphasize that most of the quantities introduced below are polynomials in $z$. We write
\[P_k(A_\la)=\sum_{|\mu|=n} a^{(k)}_\mu(n) \, \ta_\mu^\la
=\sum_{\rho}c^{(k)}_\rho \, \binom{n-|\overline{\rho}|}{m_1(\rho)} \,\ta^\la_{\tilde{\rho}}.\]

\subsection{Our equations}

Given two alphabets $A,B$ and a partition $\rho$, we have the fundamental formula
\[P_\rho(A\cup B)=\sum_{\sigma\subset \rho}P_{\rho/\sigma}(B)P_\sigma(A)\]
involving skew Hall-Littlewood functions $P_{\rho/\sigma}$~\cite[(5.5'), p. 228]{Ma}. When $B$ has only one element $b$, $P_{\rho/\sigma}(B)=0$ unless $\rho\setminus \sigma$ is a horizontal strip. In this case $P_{\rho/\sigma}(B)=\psi_{\rho/\sigma}b^{|\rho|-|\sigma|}$, where $\psi_{\rho/\sigma}=\prod_{j\in J}(1-z^{m_j(\sigma)})$ and $J$ is the set of $j$ such that $\rho\setminus \sigma$ has no node in the column $j$ and one node in the column $j+1$~\cite[(5.8') and (5.14'), p. 229]{Ma}.

Applying this classical result to the alphabet of contents $A_\la$, and writing $u_i=\la_i-i+1$, we obtain
\[P_\rho(A_{\la^{(i)}})=\sum_{\sigma\subset \rho}P_{\sigma}(A_\la)\,\psi_{\rho/\sigma}\,u_i^{|\rho|-|\sigma|},\]
summed over partitions $\sigma$ such that $\rho\setminus \sigma$ is a horizontal strip. This is the form taken by (4.12) for Hall-Littlewood symmetric functions.

When $\rho=(k)$ we obtain 
\[P_{k}(A_{\la^{(i)}})=P_{k}(A_{\la})+(1-z)\sum_{j=1}^{k-1} P_{k-j}(A_{\la})\,u_i^{j}+u_i^{k}.\]
Then relations (3.1) and (3.5) yield
\begin{align*}
\sum_{i=1}^{l(\la)+1} 
c_i(\la) \,P_k(A_{\la^{(i)}})&=P_k(A_{\la})+(1-z)\sum_{j=2}^{k-1}P_{k-j}(A_{\la})\,\sigma_j(\la)+\sigma_{k}(\la),\\
\sum_{i=1}^{l(\la)+1} 
c_i(\la) \,u_i\,P_k(A_{\la^{(i)}})&=(1-z)\sum_{j=2}^{k}P_{k-j+1}(A_{\la})\,\sigma_{j}(\la)+\sigma_{k+1}(\la),\\
\sum_{i=1}^{l(\la)+1} 
c_i(\la) \,u_i^2\,P_k(A_{\la^{(i)}})&
=P_k(A_{\la})\sigma_{2}(\la)+(1-z)\sum_{j=3}^{k+1}P_{k-j+2}(A_{\la})\,\sigma_{j}(\la)+\sigma_{k+2}(\la).
\end{align*}
Here we have omitted the details of the computation since it goes exactly as in Section 7.1. By elimination, we get immediately
\begin{align*}
\sum_{i=1}^{l(\la)+1} 
c_i(\la) \,P_k(A_{\la^{(i)}})&=P_k(A_{\la})+\sum_{i=1}^{l(\la)+1} 
c_i(\la) \,u_i\,P_{k-1}(A_{\la^{(i)}}),\\
\sum_{i=1}^{l(\la)+1} 
c_i(\la) \,u_i\,P_k(A_{\la^{(i)}})&=-nzP_{k-1}(A_{\la})+\sum_{i=1}^{l(\la)+1} 
c_i(\la) \,u_i^2\,P_{k-1}(A_{\la^{(i)}}).
\end{align*}

\subsection{Recurrence}

Applying (4.8)--(4.10) the previous relations may be written as
\begin{align*}
\sum_{|\mu|=n+1} a^{(k)}_\mu(n+1) \ta_{\mu \setminus  (1)}^\la&=\sum_{|\nu|=n} a^{(k)}_\nu(n)\, \ta_\nu^\la\\
&+\sum_{|\mu|=n+1} a_\mu^{(k-1)}(n+1)\, \sum _{r\ge 1}
r\big(m_r(\mu)+1\big)\theta^\la_{\mu \setminus (r+1)\cup (r)},\\
\sum_{|\mu|=n+1} a_\mu^{(k)}(n+1)\, &\sum _{r\ge 1}
r\big(m_r(\mu)+1\big)\theta^\la_{\mu \setminus (r+1)\cup (r)}=
-nz\sum_{|\nu|=n} a^{(k-1)}_{\nu}(n)\, \ta_\nu^\la\\
&+\sum_{|\mu|=n+1} a_\mu^{(k-1)}(n+1)\,
\Big((2n-m_1(\mu)+1) \, \ta^\la_{\mu \setminus 1}\\
&+\sum_{r,s \ge 1} rs(m_r(\mu)+1)(m_s(\mu)+\delta_{rs}+1)\, 
\ta^{\la}_{\mu \setminus (r+s+1)\cup (r,s)}
\\ &+ \sum_{r,s \ge 2} (r+s-1)(m_{r+s-1}(\mu)+1)\,
\ta^{\la}_{\mu \setminus (r,s)\cup (r+s-1)}\Big).
\end{align*}
By identification of coefficients on both sides, for any $\mu\vdash n$ we get
\begin{equation}
a^{(k)}_{\mu\cup(1)}(n+1)=a^{(k)}_\mu(n)+\sum _{r\ge 1}
rm_r(\mu)\,a^{(k-1)}_{\mu \setminus (r)\cup (r+1)}(n+1),
\end{equation}
\begin{equation}
\begin{split}
\sum _{r\ge 1}
rm_r(\mu)\,a^{(k)}_{\mu \setminus (r)\cup (r+1)}(n+1)&=
-nza^{(k-1)}_{\mu}(n)\\
&+\sum_{r,s \ge 1} rsm_r(\mu)(m_s(\mu)-\delta_{rs})\, 
a_{\mu \setminus (r,s)\cup (r+s+1)}^{(k-1)}(n+1)
\\ &+ \sum_{r,s \ge 1} (r+s-1)m_{r+s-1}(\mu)\,
a_{\mu \setminus (r+s-1)\cup (r,s)}^{(k-1)}(n+1).
\end{split}
\end{equation}

As previously, these two recurrence relations determine the coefficients $a^{(k)}_\mu(n)$ by a triple induction: on $k$, on $n$ and on the lowest part of $\mu$. Using the Lemma in the Appendix, we also obtain the following result.
\begin{theo}
We have the class expansion
\[P_k(J_1,\ldots,J_n;z)=\sum_{\rho}c^{(k)}_\rho \, \binom{n-|\overline{\rho}|}{m_1(\rho)} \, C_{\tilde{\rho}},\]
where the coefficients $c^{(k)}_\rho$ are determined by the recurrence relations
\begin{equation}
c^{(k)}_{\rho\cup (1)}=\sum_{r \ge 1} rm_r(\rho) \,c^{(k-1)}_{\rho\setminus (r) \cup(r+1)},
\end{equation}
\begin{equation}
\begin{split}
\sum_{r \ge 1} rm_r(\rho) \,c^{(k)}_{\rho\setminus (r) \cup(r+1)}&=
(2-z)|\rho|\,c^{(k-1)}_\rho+(1-z)m_1(\rho)\,c^{(k-1)}_{\rho\setminus (1)}\\
&+\sum_{r,s \ge 1} rsm_r(\rho)(m_s(\rho)-\delta_{rs})\, 
c_{\rho \setminus (r,s)\cup (r+s+1)}^{(k-1)}\\
&+\sum_{r,s \ge 1} (r+s-1)m_{r+s-1}(\rho)\,
c_{\rho \setminus (r+s-1)\cup (r,s)}^{(k-1)}.
\end{split}
\end{equation}
\end{theo}
We recover Theorems 6.1 or 7.1 by making $z=1$ or $z=0$.

\subsection{Leading terms}

The following result is proved exactly as Proposition 7.2.
\begin{prop}
The coefficients $c_{\rho}^{(k)}$ are non zero only if $|\rho|-l(\rho)= k-2i$ for some $i\ge 0$. Moreover if $|\rho|-l(\rho)=k$, they are non zero only if $m_1(\rho)=0$.
\end{prop}

For any nonnegative integer $r$, we denote by $C_r(z)$ the polynomial in $z$ defined by $C_0(z)=1$ and the recurrence formulas
\begin{align}
\notag C_r(z)&=(1-z)C_{r-1}(z)+z\,\sum_{i=0}^{r-1}C_i(z)C_{r-i-1}(z)\\
&=(1+z)C_{r-1}(z)+z\,\sum_{i=1}^{r-2}C_i(z)C_{r-i-1}(z),
\end{align}
the second relation valid for $r\ge3$. Clearly we have $C_r(0)=1$, $C_r(1)=C_r$, the Catalan number. The first values of $C_r(z)$ are given by
\begin{align*}
C_1(z)=1,\quad C_2(z)=z+1,&\quad C_3(z)=z^2+3z+1,\\ 
C_4(z)=z^3+6z^2+6z+1, &\quad C_5(z)=z^4+10z^3+20z^2+10z+1.
\end{align*}

Using known results about generating functions~\cite{S,Z}, we have
\[C_r(z)=\sum_{k=1}^{r} N(r,k) z^{k-1},\]
where the Narayana numbers $N(r,k)$ are defined by
\[ N(r,k)=\frac{1}{r} \binom{r}{k-1} \binom{r}{k}.\]
This polynomial, called Narayana polynomial, is a $z$-refinement of Catalan numbers because
\[\sum_{k=1}^{r} N(r,k)=C_r.\]
For more details we refer to~\cite{La5} and references therein.

\begin{prop}
If $|\rho|-l(\rho)=k$ and $m_1(\rho)=0$, we have
\[c_{\rho}^{(k)}=(1-z)^{l(\rho)-1}\prod_{i=1}^{l(\rho)}C_{\rho_i-1}(1-z).\]
\end{prop}
\begin{proof}
The proof is done by induction on $k$ and the lowest part $p$ of $\rho$, exactly as in Proposition 7.3. Let $p$ be the lowest part of $\rho$. Since $m_1(\rho)=0$ we have $p\ge 2$. Equation (8.4) written for $\sigma=\rho\setminus (p)\cup (p-1)$ defines $c_{\rho}^{(k)}$ inductively. First we assume $p\ge 3$. This equation takes the form
\begin{equation}
\sum_{r \ge 2} rm_r(\sigma) \,c^{(k)}_{\sigma\setminus (r) \cup(r+1)}=
(2-z)|\sigma|\,c^{(k-1)}_\sigma +\sum_{r,s \ge 2} (r+s-1)m_{r+s-1}(\sigma)\,
c_{\sigma \setminus (r+s-1)\cup (r,s)}^{(k-1)},
\end{equation}
because at the right-hand side $m_1(\sigma)=0$, the first sum vanishes due to $|\rho|-l(\rho)=k$, and the terms in the second sum with $r,s=1$ vanish for the same reason. Finally it is enough to substitute the statement into (8.6) and to prove
\begin{multline*}
\sum_{r \ge 2} rm_r(\sigma) \frac{C_r(1-z)}{C_{r-1}(1-z)}\\=(2-z)\sum_{r \ge 2} rm_r(\sigma) +(1-z)\sum_{u \ge 2} um_{u}(\sigma)\,\sum_{\begin{subarray}{c}r,s \ge 2\\r+s=u+1\end{subarray}}
\frac{C_{r-1}(1-z)C_{s-1}(1-z)}{C_{u-1}(1-z)}.
\end{multline*}
But this is an obvious consequence of (8.5) written for $1-z$. It remains to consider the case $p=2$, i.e. $\sigma=\rho\setminus (2)\cup (1)$ and $m_1(\sigma)=1$. In this case (8.4) becomes
\[c_{\sigma\setminus (1)\cup (2)}^{(k)}=(1-z)c_{\sigma\setminus (1)}^{(k-1)},\quad \textrm{i.e.} \quad c_{\rho}^{(k)}=(1-z)c_{\rho\setminus (2)}^{(k-1)},\]
because all other terms involve partitions with lowest part 1, and thus vanish due to $|\rho|-l(\rho)=k$. The statement is then a consequence of $C_1(1-z)=1$.
\end{proof}

A result of~\cite{No2} gives the leading terms of $m_\la(J_1,\ldots,J_n)$, with $m_\la$ a monomial symmetric function. As a corollary of Proposition 8.3 we may obtain a similar, but weaker, result for the symmetric function
\[p_{r,s}=\sum_{\begin{subarray}{c}|\la|=r\\
l(\la)=s\end{subarray}}m_\la.\]
We shall make use of the following property~\cite[Theorem 2]{La5}
\begin{equation}
C_r(1-z)=\frac{1}{r+1}P_r(1^{r+1};z).
\end{equation}
 
\begin{prop}
The leading terms of 
\[p_{r,s}(J_1,\ldots,J_n)=\sum_{\rho}c^{(r,s)}_\rho \, \binom{n-|\overline{\rho}|}{m_1(\rho)} \, C_{\tilde{\rho}},\]
are obtained for $|\rho|-l(\rho)=r$ and $m_1(\rho)=0$. Their coefficients are
\[c_{\rho}^{(r,s)}=\sum_{
\begin{subarray}{c}1\le s_i\le \rho_i-1\\s_1+\cdots+s_{l(\rho)}=s\end{subarray}} 
\prod_{i= 1}^{l(\rho)}N(\rho_i-1,s_i).\]
\end{prop}
\begin{proof} Given an alphabet $X$ and using $\la$-ring notations (see Section 6.2), it is well known~(\cite[(2.10), p. 209]{Ma} and \cite[p. 240]{La3}) that
\begin{equation}
(1-z)P_r(X;z)=h_r[(1-z)X^\dag]=\sum_{|\la|=r} (1-z)^{l(\la)}m_\la(X)=\sum_{s=1}^r (1-z)^s p_{r,s}(X).
\end{equation}
Given a partition $\rho$ with $|\rho|-l(\rho)=r$ and $m_1(\rho)=0$, let us compare the leading coefficient on both sides. For $(1-z)P_r(J_1,\ldots,J_n;z)$, it is given by Proposition 8.3 as
\[\prod_{i=1}^{l(\rho)}\Big((1-z)C_{\rho_i-1}(1-z)\Big).\]
In view of (8.7) and (8.8) this may be written as
\[\prod_{i=1}^{l(\rho)}\Big((1-z)\frac{1}{\rho_i}P_{\rho_i-1}(1^{\rho_i};z)\Big)
=\prod_{i=1}^{l(\rho)}\sum_{m=1}^{\rho_i-1}(1-z)^m \frac{1}{\rho_i}p_{\rho_i-1,m}(1^{\rho_i}).\]
By comparison we obtain that the leading coefficient for $p_{r,s}(J_1,\ldots,J_n)$ is 
\[\sum_{
\begin{subarray}{c}1\le s_i\le \rho_i-1\\s_1+\cdots+s_{l(\rho)}=s\end{subarray}} 
\prod_{i= 1}^{l(\rho)}
\frac{1}{\rho_i}p_{\rho_i-1,s_i}(1^{\rho_i})\]
But~\cite[Example 1.2.19, p. 33]{Ma} we have
\[p_{r,s}(1^n)=\sum_{a+b=r}(-1)^{a-s}\binom{a}{s}\binom{n}{a}\binom{n+b-1}{b}=\binom{r-1}{s-1}\binom{n}{s},\]
and
\[\frac{1}{r+1}\binom{r-1}{s-1}\binom{r+1}{s}=N(r,s).\]
\end{proof}
\noindent\textit{Examples:} For $s=1$ we recover that the leading term of $p_{k,1}=p_k$ is obtained for $\rho=(k+1)$ with coefficient 1. For $s=2$ the leading terms of $p_{k,2}$ are obtained for $\rho=(k+1)$ with coefficient $\binom{k}{2}$, and for $|\rho|=(k+2)$ and $l(\rho)=2$ with coefficient 1. For $s=3$ the leading terms of $p_{k,3}$ are obtained for $\rho=(k+1)$ with coefficient $N(k,3)$, for $|\rho|=(k+2)$ and $l(\rho)=2$ with coefficient $\binom{\rho_1-1}{2}+\binom{\rho_2-1}{2}$, and for $|\rho|=(k+3)$ and $l(\rho)=3$ with coefficient 1.

\section{Generating function}

We consider the generating function of the coefficients $c_{\rho}^{(k)}$ defined by Theorem 8.1. As for $z=1$, we write it under the form
\[\phi_\rho(t)=\sum_{k\ge 0} c_{\rho}^{(k)} \, \frac{t^k}{k!}.\]
We present two methods to study this generating function. 
 
\subsection{First method: differential system}

This method is strictly parallel to the case of power sums, given in Section 6.2. However the situation is much more complicated, and strong difficulties are encountered, which makes this approach unefficient.

As in Section 6.2 we define
\[\Phi_w(t;X)= \sum_{|\rho|=w}\phi_\rho(t)\, z_\rho^{-1}p_\rho(X)
=\sum_{|\rho|=w}\psi_\rho(t)\, s_\rho(X).\]
The following result is proved exactly as Proposition 6.2.
\begin{prop}
The recurrence relations (8.3)--(8.4) are equivalent with
\begin{equation}
\frac{d}{dt}\,\frac{\partial}{\partial p_1}\Phi_w=\mathbf{E}\, \Phi_w,
\qquad \frac{d}{dt}\,\mathbf{E}\,\Phi_{w+1}= \mathbf{D}\,\Phi_{w+1}+(2-z)w\Phi_w+(1-z)p_1\Phi_{w-1}.
\end{equation}
\end{prop}

Applying (4.7), by identification of coefficients of Schur functions on both sides, Proposition 9.1 may be written equivalently as
\begin{equation}
\begin{split}
\sum_{i=1}^{l(\rho)+1}\frac{d}{dt} \,\psi_{\rho^{(i)}}(t)&=
\sum_{i=1}^{l(\rho)+1}(\rho_i-i+1) \,\psi_{\rho^{(i)}}(t),\\
\sum_{i=1}^{l(\rho)+1}(\rho_i-i+1)\frac{d}{dt} \,\psi_{\rho^{(i)}}(t)&=
\sum_{i=1}^{l(\rho)+1}(\rho_i-i+1)^2 \,\psi_{\rho^{(i)}}(t)\\
&+(2-z)|\rho|\psi_\rho(t)+(1-z)\sum_{i=1}^{l(\rho)}
\psi_{\rho_{(i)}}(t).
\end{split}
\end{equation}
This first order (overdeterminate) differential system must be solved with the initial conditions $\psi_{\rho}(0)=\delta_{\rho,(0)}$, due to $c^{(0)}_\rho=\delta_{\rho,(0)}$.

However, in spite of its very simple structure, a general solution of the differential system (9.2) is as yet unknown. By elementary means we have
\begin{align*}
\psi_{\rho}&=\psi_{\rho}\arrowvert_{z=1}, \quad \textrm{if} \quad |\rho|=2,\\
\psi_{\rho}&=(2-z)\, \psi_{\rho}\arrowvert_{z=1}, \quad \textrm{if} \quad |\rho|=3.
\end{align*}
But a general formula for $\psi_{\rho}(t)$ is lacking even when $\rho$ is a hook $(r,1^s)$. The case of hooks is only known for $0\le s\le 3$, in which cases the structure of $\psi_{(r,1^s)}(t)$ is already very messy~\cite{La7}.

\subsection{The case $z=1$ revisited}

The difficulties encountered to solve the differential system (9.2) lead us to a very different approach. We shall revisit the case $z=1$ where, according to Theorem 6.4, the generating function $\hat{\phi}_\rho:=\phi_\rho\arrowvert_{z=1}$ may be written as
\begin{equation*}
|\rho|!\, \hat{\phi}_\rho(t)=
e^{(1-|\rho|)t}\,(e^{t}-1)^{|\rho|-2}\,\prod_{i\ge 1}(e^{it}-1)^{m_i(\rho)}.
\end{equation*}
Let $\epsilon_\rho=(-1)^{|\rho|-l(\rho)}$. Expanding the right-hand side we obtain (up to a constant term obtained by $\hat{\phi}_\rho(0)=0$)
\[|\rho|!\, \hat{\phi}_\rho(t)=\sum_{k=1}^{|\rho|-1}
f^{(k)}_\rho (e^{kt}+\epsilon_\rho e^{-kt}),\]
where the coefficients $f_\rho^{(k)}$ are given as follows. 

For any integer $r\ge 0$, denote by $\mathbf{I}_r$ the family of nonnegative integers $I=(i_0,i_1,i_2,\ldots)$ linked by $i_0+\sum_{u\ge 1}ui_u=r$. For $k=|\rho|-1-r$ we have
\[f^{(k)}_{\rho}=\sum_{I\subset \mathbf{I}_r}(-1)^{|I|}\binom{|\rho|-2}{i_0}
\prod_{u \ge 1}\binom{m_u(\rho)}{i_u}.\]
For instance $\mathbf{I}_0= 0$, $\mathbf{I}_1= \{(1,0), (0,1)\}$, $\mathbf{I}_2=\{(2,0,0),(1,1,0),(0,2,0),(0,0,1)\}$ and $\mathbf{I}_3=\{(3,0,0,0),(2,1,0,0),(1,2,0,0),(0,3,0,0),(0,1,1,0),(1,0,1,0),(0,0,0,1)\}$ so that we have
\begin{align*}
f^{(|\rho|-1)}_{\rho}&=1, \qquad f^{(|\rho|-2)}_{\rho}=-(|\rho|-2+m_1(\rho)), \\
f^{(|\rho|-3)}_{\rho}&=\binom{|\rho|-2}{2}+
(|\rho|-2)m_1(\rho)+\binom{m_1(\rho)}{2}-m_2(\rho),\\
f^{(|\rho|-4)}_{\rho}&=-\binom{|\rho|-2}{3}-
\binom{|\rho|-2}{2}m_1(\rho)-(|\rho|-2)\binom{m_1(\rho)}{2}-\binom{m_1(\rho)}{3}\\
&\hspace{4cm}+m_1(\rho)m_2(\rho)+(|\rho|-2)m_2(\rho)-m_3(\rho).
\end{align*}

For a better display, we write the multiplicities of $\rho$ as a set of indeterminates $\mathbf{m}=(m_1,m_2,\ldots)$ linked by $\sum_{i\ge 1} im_i=|\rho|$. For $I=(i_0,i_1,i_2,\ldots)$ we denote
\[M_I(w,\mathbf{m})=(-1)^{|I|}\binom{w-2}{i_0}\prod_{u \ge 1}\binom{m_u}{i_u}.\]
Then we have for $k=|\rho|-1-r$,
\begin{equation}
f^{(k)}_{\rho}=\sum_{I\subset \mathbf{I}_r}M_I(|\rho|,\mathbf{m}).
\end{equation}

\subsection{Second method: expansion of $\phi_\rho$}

Inspired by the case $z=1$, we try to generalize this expansion of $\phi_\rho\arrowvert_{z=1}$ for arbitrary $z$. Namely for any partition $\rho$, we look for an expression of $\phi_\rho$ as 
\begin{equation*}
|\rho|!\, \phi_\rho(t)=\sum_{k=1}^{|\rho|-1}
(f^{(k)}_{\rho}(t)e^{kt}+\epsilon_\rho f^{(k)}_{\rho}(-t)e^{-kt}),
\end{equation*}
up to a constant term known by $\phi_\rho(0)=0$. Here $f_\rho^{(k)}$ is no longer a constant but a function of $t$.

Let $|\rho|=w$. By substitution in the definition of $\Phi_w$, we get
\[w! \Phi_w(t;X)= \sum_{k=1}^{w-1}\sum_{|\rho|=w}(f^{(k)}_{\rho}(t)e^{kt}+\epsilon_\rho f^{(k)}_{\rho}(-t)e^{-kt})\, \widehat{p}_\rho(X).\]
Applying (4.4), we obtain
\begin{multline*}
w!\frac{d}{dt}\,\frac{\partial}{\partial p_1}\Phi_w=\\
\sum_{k=1}^{w-1}\sum_{|\rho|=w}
\left(\Big(k f^{(k)}_{\rho}(t)+ \frac{d}{dt}f^{(k)}_{\rho}(t)\Big)e^{kt}
-\epsilon_\rho\Big(k f^{(k)}_{\rho}(-t)+ \frac{d}{dt}f^{(k)}_{\rho}(-t)\Big)e^{-kt}\right)\,\widehat{p}_{\rho\setminus (1)},
\end{multline*}
\begin{align*}
w!\mathbf{E}\,\Phi_w=&\sum_{k=1}^{w-1}\sum_{|\rho|=w}(f^{(k)}_{\rho}(t)e^{kt}+\epsilon_\rho f^{(k)}_{\rho}(-t)e^{-kt}) \,\sum_{r\ge 1}r(m_{r}(\rho)+1)\, \widehat{p}_{\rho \setminus (r+1)\cup (r)}\\
(w+1)!\mathbf{D} \,\Phi_{w+1}=&\sum_{k=1}^{w}\sum_{|\rho|=w+1}(f^{(k)}_{\rho}(t)e^{kt}+\epsilon_\rho f^{(k)}_{\rho}(-t)e^{-kt})\\& \Big(\sum_{r,s\ge 1} (r+s-1)(m_{r+s-1}(\rho)+1)\, \widehat{p}_{\rho \setminus (r,s)\cup (r+s-1)}\\
&+\sum_{r,s \ge 1} rs(m_r(\rho)+1)(m_s(\rho)+1+\delta_{rs})\,
\widehat{p}_{\rho \setminus (r+s+1)\cup (r,s)}\Big).
\end{align*}

Finally if we identify the coefficients of products $e^{kt}\,\widehat{p}_\rho$ on both sides of equations (9.1), we obtain for any $k=1,\ldots,w$,
\begin{equation}
k f^{(k)}_{\rho\cup (1)}+ \frac{d}{dt}f^{(k)}_{\rho\cup (1)}=\sum_{r \ge 1} rm_r(\rho) \,f^{(k)}_{\rho\setminus (r) \cup(r+1)},
\end{equation}
\begin{equation}
\begin{split}
\sum_{r \ge 1} rm_r(\rho) \,\Big(
k f^{(k)}_{\rho\setminus (r) \cup(r+1)}+ \frac{d}{dt}f^{(k)}_{\rho\setminus (r) \cup(r+1)}\Big)&=w(w+1)\Big(
(2-z)\,f^{(k)}_\rho+(1-z)m_1(\rho)\,f^{(k)}_{\rho\setminus (1)}\Big)\\
&+\sum_{r,s \ge 1} rsm_r(\rho)(m_s(\rho)-\delta_{rs})\, 
f^{(k)}_{\rho \setminus (r,s)\cup (r+s+1)}\\
&+\sum_{r,s \ge 1} (r+s-1)m_{r+s-1}(\rho)\,
f^{(k)}_{\rho \setminus (r+s-1)\cup (r,s)}.
\end{split}
\end{equation}
The coefficients $f^{(k)}_\rho$ are defined inductively as solutions of these equations. Indeed let us assume that $f^{(k)}_\rho$ is known for any $\rho$ with $|\rho|\le w$ and any $k\le w-1$. Equations (9.4)--(9.5) define an overdetermined linear system, the solutions of which are the coefficients $\{f^{(k)}_\rho, |\rho|=w+1\}$. 

Unfortunately a general solution of this linear system is as yet unknown. Empirically a (unique) solution does exist. Tables giving $\phi_\rho(t)$ for $|\rho|\le 14$ are available on a web page~\cite{W2}. First values are given below. 

\subsection{Examples}
 
We use the convention 
\[
e^{kt}\pm e^{-kt}=e^{kt}+\epsilon_\rho \,e^{-kt},\qquad
e^{kt}\mp e^{-kt}=e^{kt}-\epsilon_\rho \,e^{-kt}.
\]
We denote by $R_k(z)$ the polynomial in $z$ defined by $R_1(z)=1$ and
\[R_k(z)=\frac{1}{(k-1)!}\prod_{j=1}^{k-1} (k-jz).\]
For instance $R_2(z)=2-z$ and $R_3(z)=(3-z)(3-2z)/2$.
 
We list the first values of $\phi_{\rho}$. All formulas are given up to a constant term obtained by $\phi_{\rho}(0)=0$ (this constant term is 0 when $\epsilon_\rho=-1$). For $|\rho|=2$ we have
\[\phi_{\rho}(t)=\phi_{\rho}(t)\arrowvert_{z=1}= \frac{1}{2}(e^{t}\pm e^{-t}).\]
For $|\rho|=3$ we have
\[\phi_{\rho}(t)=(2-z)\,
\phi_{\rho}(t)\arrowvert_{z=1}= \frac{1}{6}R_2(z) \big(e^{2t}\pm e^{-2t}-(m_1+1)(e^{t}\pm e^{-t})\big).\]
For $|\rho|=4$ we have
\begin{align*}
4! \,\phi_{\rho}(t)&= R_3(z)(e^{3t}\pm e^{-3t})-R_2(z)(2-z)(m_1+2)(e^{2t}\pm e^{-2t})+(e^{t}\pm e^{-t})\\
&\times\left((2m_1+1)(z^2-\frac{5}{2}(z-1))
+\binom{m_1}{2}(z^2-\frac{11}{2}(z-1))
-m_2(z^2+\frac{1}{2}(z-1))\right).
\end{align*}
For $|\rho|=5$ we have
\begin{align*}
5! \,\phi_{\rho}(t)&= R_4(z)(e^{4t}\pm e^{-4t})-R_3(z)(2-z)(m_1+3)(e^{3t}\pm e^{-3t})
+R_2(z)(e^{2t}\pm e^{-2t})\\
&\times\left((3m_1+3)(z^2-\frac{28}{9}(z-1))
+\binom{m_1}{2}(z^2-\frac{16}{3}(z-1))
-m_2(z^2+\frac{4}{3}(z-1))\right)\\
&-(2-z)(e^{t}\pm e^{-t}) \left((3m_1+1)(z^2-\frac{7}{6}(z-1))
+3\binom{m_1}{2}(z^2-\frac{17}{6}(z-1))\right.\\
&\left.+\binom{m_1}{3}(z^2-\frac{19}{2}(z-1))
-(m_1m_2+3m_2-m_3)(z^2+\frac{1}{2}(z-1))\right).
\end{align*}

In the previous examples, all $f_{\rho}^{(k)}$ do not depend on $t$. But for $|\rho|=6$ a $t$-component appears at $k=1$. We have 
\begin{align*}
6! \,\phi_{\rho}(t)&= R_5(z)(e^{5t}\pm e^{-5t})-R_4(z)(2-z)(m_1+4)(e^{4t}\pm e^{-4t})
+R_3(z)(e^{3t}\pm e^{-3t})\\
&\times\left((4m_1+6)(z^2-\frac{27}{8}(z-1))
+\binom{m_1}{2}(z^2-\frac{21}{4}(z-1))
-m_2(z^2+\frac{9}{4}(z-1))\right)\\
&-R_2(z)(2-z)(e^{2t}\pm e^{-2t}) \left((6m_1+4)(z^2-2z+2)
+4\binom{m_1}{2}(z^2-\frac{11}{3}(z-1))\right.\\
&\left.+\binom{m_1}{3}(z^2-\frac{26}{3}(z-1))
-(m_1m_2+4m_2-m_3)(z^2+\frac{4}{3}(z-1))\right)\\
&+(e^{t}\pm e^{-t})\sum_{I\subset \mathbf{I}_4} 
M_I(6,\mathbf{m})\Big(z^4+a_Iz^2(z-1)+b_I(z-1)^2\Big)\\
&+t(e^{t}\mp e^{-t})\, (z^2-1)(2z-1)\sum_{I\subset \mathbf{I}_4} M_I(6,\mathbf{m})c_I,
\end{align*}
with $a_I$, $b_I$ and $c_I$ listed below.
\vspace{0.3 cm}\\
\footnotesize
\begin{tabular}{|c|c|c|c|}
\hline
$I$ & $a_I$ & $b_I$ & $c_I$ \\ \hline
(4,0,0,0,0) (3,1,0,0,0) &-7/2 & 7/4& 0\\ \hline
(0,4,0,0,0) &-245/12 & 2035/24 & 5/4 \\ \hline
(1,3,0,0,0) &-91/16 & 591/32 & -5/8 \\ \hline
(2,2,0,0,0) &-47/6 & 37/12 & 5/24 \\ \hline
(2,0,1,0,0) & 0 &0 & -5/24 \\ \hline
\end{tabular}
\begin{tabular}{|c|c|c|c|}
\hline
$I$ & $a_I$ & $b_I$ & $c_I$ \\ \hline
(0,2,1,0,0) (0,0,0,0,1)& -33/4 &-77/8 &-5/4\\ \hline
(1,1,1,0,0) & -263/48 & -197/96 & 5/8 \\ \hline
(0,0,2,0,0) & 115/12 & 235/24 & 5/4 \\ \hline
(0,1,0,1,0) & -5/12 & 115/24 & 5/4 \\ \hline
(1,0,0,1,0) & -11/16 &-49/32 & -5/8 \\ \hline
\end{tabular}
\normalsize
\vspace{0.3 cm}\\
Similarly a $t$-component appears for $|\rho|=7$ at $k=1$. We have
\begin{align*}
7! \,\phi_{\rho}(t)&= R_6(z)(e^{6t}\pm e^{-6t})-R_5(z)(2-z)(m_1+5)(e^{5t}\pm e^{-5t})
+R_4(z)(e^{4t}\pm e^{-4t})\\
&\times\left(5(m_1+2)(z^2-\frac{88}{25}(z-1))
+\binom{m_1}{2}(z^2-\frac{26}{5}(z-1))
-m_2(z^2+\frac{16}{5}(z-1))\right)\\
&-R_3(z)(2-z)(e^{3t}\pm e^{-3t}) \left(10(m_1+1)(z^2-\frac{99}{40}(z-1))
+5\binom{m_1}{2}(z^2-\frac{81}{20}(z-1))\right.\\
&\left.+\binom{m_1}{3}(z^2-\frac{33}{4}(z-1))
-(m_1m_2+5m_2-m_3)(z^2+\frac{9}{4}(z-1))\right)\\
&+R_2(z)(e^{2t}\pm e^{-2t})\sum_{I\subset \mathbf{I}_4} 
M_I(7,\mathbf{m})\Big(z^4+A_Iz^2(z-1)+B_I(z-1)^2\Big)\\
&+(e^{t}\pm e^{-t})(2-z)\sum_{I\subset \mathbf{I}_5} 
M_I(7,\mathbf{m})\Big(z^4+a_Iz^2(z-1)+b_I(z-1)^2\Big)\\
&+t(e^{t}\mp e^{-t})\, (2-z)(z^2-1)(2z-1)\sum_{I\subset \mathbf{I}_5} M_I(7,\mathbf{m})c_I,
\end{align*}
with $A_I$, $B_I$, $a_I$, $b_I$ and $c_I$ listed below.
\vspace{0.3 cm}\\
\footnotesize
\begin{tabular}{|c|c|c|}
\hline
$I$ & $A_I$ & $B_I$ \\ \hline
(4,0,0,0,0) (3,1,0,0,0) &-47/10 &22/5\\ \hline
(0,4,0,0,0) &-47/3 & 212/3\\ \hline
(1,3,0,0,0) &-172/15&304/15 \\ \hline
(2,2,0,0,0) &-88/15& 136/15 \\ \hline
(2,0,1,0,0) &-71/30& -74/15\\ \hline
\end{tabular}
\begin{tabular}{|c|c|c|}
\hline
$I$ & $A_I$ & $B_I$  \\ \hline
(0,2,1,0,0) (0,0,0,0,1)&-26/3&-40/3\\ \hline
(1,1,1,0,0)&-4/15 & -32/15  \\ \hline
(0,0,2,0,0)&37/3& 44/3\\ \hline
(0,1,0,1,0)&11/6& 2/3\\ \hline
(1,0,0,1,0)&-67/15 &-116/15 \\ \hline
\end{tabular}
\vspace{0.3 cm}\\
\begin{tabular}{|c|c|c|c|}
\hline
$I$ & $a_I$\rule{0cm}{0.4cm} & $b_I$ & $c_I$ \\ \hline
(5,0,0,0,0,0) (4,1,0,0,0,0)&-19/10 & 11/20 &0\\ \hline
(0,5,0,0,0,0)& 4371/10 & 6271/20 & 77/4\\ \hline
(1,4,0,0,0,0)& -4167/50 & 3573/100 & -119/20\\ \hline
(3,2,0,0,0,0)& 0 & 0 & 7/40\\ \hline
(2,3,0,0,0,0)& 0 & 0 & 7/8\\ \hline
(3,0,1,0,0,0)& -9/2 & 3/4 & -7/40\\ \hline
(0,3,1,0,0,0)& -1281/20 & -3521/40 &-49/4 \\ \hline
(2,1,1,0,0,0)& 0 & 0 & -7/40 \\\hline
(1,2,1,0,0,0)& 0 & 0 &63/20\\ \hline
\end{tabular}
\begin{tabular}{|c|c|c|c|}
\hline
$I$ & $a_I$\rule{0cm}{0.4cm} & $b_I$ & $c_I$ \\ \hline
(1,0,2,0,0,0)& 0 & 0 &-7/20\\ \hline
(0,1,2,0,0,0)& -163/10 & 717/20 &21/4 \\ \hline
(0,0,1,1,0,0)& -257/40 & -1137/80 & -7/4\\ \hline
(2,0,0,1,0,0)& -163/20 & 17/40 & -7/40\\ \hline
(0,2,0,1,0,0)& -3623/40 & 5417/80 &35/4 \\ \hline
(1,1,0,1,0,0)& 1249/100 & -1331/200 &-7/4 \\ \hline
(1,0,0,0,1,0)& -1363/100 & 697/200 &7/20 \\ \hline
(0,1,0,0,1,0)& 1849/20 &-2091/40 &-21/4 \\ \hline
(0,0,0,0,0,1)& -593/20 & 907/40 &7/4 \\ \hline
\end{tabular}
\normalsize
\vspace{0.3 cm}\\

\subsection{Constant terms}

Let $|\rho|=w$. In the previous examples, it appears empirically that the solutions of the linear system (9.4)--(9.5) take the form
\[f^{(k)}_{\rho}=R_k(z)\,\sum_{I\subset \mathbf{I}_r}M_I(w,\mathbf{m})\,f_I(z,w;t),\]
with $k=w-1-r$ and $f_I$ a function of $z$, $w$ and $t$. Of course for $z=1$ since we have $R_k(1)=1$, this is in accordance with (9.3) and we have $f_I(1,w;t)=1$.

It also appears empirically that for $w/2-2< k\le w-1$, the coefficient $f^{(k)}_{\rho}$ does not depend on $t$. It may be written as
\begin{equation}
f^{(k)}_{\rho}=R_k(z)\,\sum_{I\subset \mathbf{I}_r}M_I(w,\mathbf{m})\,f_I(z,w).
\end{equation}
Since $k=w-1-r$ such cases correspond to $0\le r <w/2 +1$.

In this situation the linear system (9.4)--(9.4) takes the form 
\begin{equation}
k f^{(k)}_{\rho\cup (1)}=\sum_{u \ge 1} um_u(\rho) \,f^{(k)}_{\rho\setminus (u) \cup(u+1)},
\end{equation}
\begin{equation}
\begin{split}
k^2 f^{(k)}_{\rho\cup (1)}&=w(w+1)\Big(
(2-z)\,f^{(k)}_\rho+(1-z)m_1(\rho)\,f^{(k)}_{\rho\setminus (1)}\Big)\\
&+\sum_{u,v \ge 1} uvm_u(\rho)(m_v(\rho)-\delta_{uv})\, 
f^{(k)}_{\rho \setminus (u,v)\cup (u+v+1)}\\
&+\sum_{u,v \ge 1} (u+v-1)m_{u+v-1}(\rho)\,
f^{(k)}_{\rho \setminus (u+v-1)\cup (u,v)}.
\end{split}
\end{equation}
We may substitute the expression (9.6) in these equations and identify the coefficients of $M_I(w+1,\mathbf{m})$ on both sides. Doing so, we obtain a linear system between the $f_I, I\subset \mathbf{I}_r$.

Let us make this method explicit for (9.7). By substitution of (9.6) it becomes
\begin{align*}
&(w-r)\sum_{I\subset \mathbf{I}_r}(-1)^{|I|}\binom{w-1}{i_0}\left(\binom{m_1}{i_1}+\binom{m_1}{i_1-1}\right)\prod_{a \ge 2}\binom{m_a}{i_a}\,f_I(z,w+1)=\\
&\sum_{u \le r} um_u \sum_{I\subset \mathbf{I}_r}(-1)^{|I|}\binom{w-1}{i_0}\binom{m_u-1}{i_u}\left(\binom{m_{u+1}}{i_{u+1}}+\binom{m_{u+1}}{i_{u+1}-1}\right)\prod_{a \neq u,u+1}\binom{m_a}{i_a}f_I(z,w+1)\\
&+(w-\sum_{u \le r}um_u) \sum_{I\subset \mathbf{I}_r}(-1)^{|I|}\binom{w-1}{i_0}\,\prod_{a \ge1}\binom{m_a}{i_a}\,f_I(z,w+1).
\end{align*}
The right-hand side may be transformed to
\begin{align*}
&\sum_{u \le r} um_u \sum_{I\subset \mathbf{I}_r}(-1)^{|I|}\binom{w-1}{i_0}\,\prod_{a \neq u,u+1}\binom{m_a}{i_a}\binom{m_u-1}{i_u}\binom{m_{u+1}}{i_{u+1}-1}\,f_I(z,w+1)\\
&-\sum_{u \le r} um_u \sum_{I\subset \mathbf{I}_r}(-1)^{|I|}\binom{w-1}{i_0}\,\prod_{a \neq u,u+1}\binom{m_a}{i_a}\binom{m_u-1}{i_u-1}\binom{m_{u+1}}{i_{u+1}}\,f_I(z,w+1)\\
&+w\sum_{I\subset \mathbf{I}_r}(-1)^{|I|}\binom{w-1}{i_0}\,\prod_{a \ge1}\binom{m_a}{i_a}\,f_I(z,w+1).
\end{align*}
Equivalently
\begin{align*}
&\sum_{I\subset \mathbf{I}_r}\sum_{1 \le u \le r} u(i_u+1) (-1)^{|I|}\binom{w-1}{i_0}\,\prod_{a \neq u,u+1}\binom{m_a}{i_a}\binom{m_u}{i_u+1}\binom{m_{u+1}}{i_{u+1}-1}\,f_I(z,w+1)\\
&-\sum_{I\subset \mathbf{I}_r}\sum_{1 \le u \le r} ui_u (-1)^{|I|}\binom{w-1}{i_0}\,\prod_{a \neq u,u+1}\binom{m_a}{i_a}\binom{m_u}{i_u}\binom{m_{u+1}}{i_{u+1}}\,f_I(z,w+1)\\
&+w\sum_{I\subset \mathbf{I}_r}(-1)^{|I|}\binom{w-1}{i_0}\,\prod_{a \ge1}\binom{m_a}{i_a}\,f_I(z,w+1).
\end{align*}
Since $\sum_{1\le u \le r} ui_u=r-i_0$ we may simplify both sides. Doing so, (9.7) is equivalent with
\begin{align*}
&(w-r)\sum_{I\subset \mathbf{I}_r}(-1)^{|I|}\binom{w-1}{i_0}\binom{m_1}{i_1-1}\prod_{a \ge 2}\binom{m_a}{i_a}\,f_I(z,w+1)=\\
&i_0\,\sum_{I\subset \mathbf{I}_r}(-1)^{|I|}\binom{w-1}{i_0}\,\prod_{a \ge1}\binom{m_a}{i_a}\,f_I(z,w+1)\\
&+\sum_{I\subset \mathbf{I}_r}\sum_{1 \le u \le r} u(i_u+1) (-1)^{|I|}\binom{w-1}{i_0}\,\prod_{a \neq u,u+1}\binom{m_a}{i_a}\binom{m_u}{i_u+1}\binom{m_{u+1}}{i_{u+1}-1}\,f_I(z,w+1).
\end{align*}
Finally by identification of coefficients in $\binom{m_1}{i_1}\ldots\binom{m_r}{i_r}$ on both sides, we get the linear relation
\begin{equation}
(w-r)f_{I\cup(1)}(z,w+1)=(w-i_0-1)f_{I\cup(0)}(z,w+1)
-\sum_{1 \le u \le r-1} ui_u\,f_{I\setminus(u)\cup(u+1)}(z,w+1),
\end{equation}
valid for any $I=(i_0,i_1,\ldots,i_{r-1})\subset \mathbf{I}_{r-1}$, where we denote $I\setminus(u)\cup(v)=(i_0,\ldots,i_u-1,\ldots,i_v+1,\ldots)$. 

In particular for $I=(r-1,0,\ldots,0)$ we get
\[f_{(r-1,1,0,\ldots,0)}(z,w)=f_{(r,0,\ldots,0)}(z,w).\]
And for $I=(r-2,1,0,\ldots,0)$ we have
\[(w-r-1)f_{(r-2,2,0,\ldots,0)}(z,w)=(w-r)f_{(r-1,1,0,\ldots,0)}(z,w)
-f_{(r-2,0,1,0,\ldots,0)}(z,w).\]

A similar, but much more involved, transformation may be performed with (9.8). It yields
\begin{equation}
\begin{split}
(w-r)^2f_{I\cup(1)}(z,w+1)&=-w(w+1)
(2-z)\frac{w-i_0-1}{w-1}f_I(z,w)\\
&+w(w+1)(1-z)i_1\,\frac{(w-i_0-1)(w-i_0-2)}{(w-1)(w-2)}f_{I\setminus (1)}(z,w-1)\\
&+(2w-2r+3)(w-i_0-1)f_{I\cup(0)}(z,w+1)\\
&+\sum_{1\le u,v \le r-1} uvi_u(i_v-\delta_{uv})\, 
f_{I \setminus (u,v)\cup (u+v+1)}(z,w+1)\\
&+\sum_{1 \le u,v \le r-1} (u+v-1)i_{u+v-1}\,
f_{I \setminus (u+v-1)\cup (u,v)}(z,w+1),
\end{split}
\end{equation}
valid for any $I=(i_0,i_1,\ldots,i_{r-1})\subset \mathbf{I}_{r-1}$, where we denote $I\setminus(u,u)\cup(v,v)=(i_0,\ldots,i_u-2,\ldots,i_v+2,\ldots)$. 

Equations (9.9) and (9.10), written for any $I=(i_0,i_1,\ldots,i_{r-1})\subset \mathbf{I}_{r-1}$, form an overdetermined linear system between the $f_I, I\subset \mathbf{I}_r$. Empirically a (unique) solution does exist.

The values for $r\le 4$ are as follows. For $r=0$ we have 
$f_{(0)}(z,w)=1$. For $r=1$ we get $f_{(1,0)}(z,w)=f_{(0,1)}(z,w)=(2-z)$.

For $r=2$ we obtain $f_I(z,w)=z^2+a_I(w)(z-1)$ with $a_I(w)$ given by
\[a_{(2,0,0)}(w)=a_{(1,1,0)}(w)=-2\frac{(w-3)(2w-3)}{(w-2)^2},\]
\[a_{(0,2,0)}(w)=-\frac{5w-9}{w-2}, \qquad a_{(0,0,1)}(w)=\frac{(w-3)^2}{w-2}.\]
For $r=3$ we have $f_I(z,w)=(2-z)\big(z^2+a_I(w)(z-1)\big)$ with $a_I(w)$ given by
\begin{align*}
a_{(3,0,0,0)}(w)&=a_{(2,1,0,0)}(w)=-2\frac{(w-4)^2(2w-3)}{(w-2)
(w-3)^2},\\
a_{(1,2,0,0)}(w)&=-\frac{(w-4)(5w-8)}{(w-2)(w-3)}, \qquad
a_{(0,3,0,0)}(w)=-\frac{7w-16}{w-3},\\
a_{(0,1,1,0)}(w)&=a_{(1,0,1,0)}(w)=a_{(0,0,0,1)}(w)=
\frac{(w-4)^2}{w-3}.
\end{align*}

Unfortunately as soon as $r\ge 4$, the expression of $f_I(z)$ becomes very messy. For $r=4$ we have 
\[f_I(z,w)=z^4+\frac{a_I(w)}{c_I(w)}z^2(z-1)+\frac{b_I(w)}{c_I(w)}(z-1)^2,\] 
where for $w\ge7$, $a_I(w)$, $b_I(w)$ and $c_I(w)$ are given by the following table.
\footnotesize
\vspace{0.3 cm}\\
\begin{tabular}{|c|c|c|}
\hline
$I$ & $a_I$ & $b_I$ \\ \hline
(4,0,0,0,0) & & \\
(3,1,0,0,0) & $-4(w-5)(2w-5)(w^3-10w^2+32w-30)$
& $4(w-5)^3(2w-5)(2w-3)$  \\ \hline
(0,4,0,0,0) & $-2(7w^3-84w^2+312w-375)$
& $43w^3-485w^2+1677w-1875$ \\ \hline
(1,3,0,0,0) & $-(w-5)(11w^3-111w^2+330w-300)$ &
$2(w-5)(14w^3-145w^2+426w-375)$ \\ \hline
(2,2,0,0,0) & $-9w^5+171w^4-1249w^3+4361w^2-7214w+4500$
& $2(w-5)^2(10w^3-99w^2+274w-225)$ \\ \hline
(2,0,1,0,0) & $(w-5)(w^5-23w^4+193w^3-745w^2+1334w-900)$
& $-2(w-5)^3(2w^3-19w^2+52w-45)$ \\ \hline
(0,2,1,0,0) & &  \\
(0,0,0,0,1) & $(w-5)(w^2-16w+50)$ & $-5(w-5)^3$ \\ \hline
(1,1,1,0,0) & $w^4-22w^3+167w^2-506w+500$ &
$-2(w-5)^2(2w^2-17w+25)$ \\ \hline
(0,0,2,0,0) & $2(w-5)(w^3-12w^2+51w-75)$
& $(w-5)^3(w^2-6w+15)$  \\ \hline
(0,1,0,1,0) & $w^4-23w^3+188w^2-636w+750$
& $-(w-5)^2(4w^2-39w+75)$  \\ \hline
(1,0,0,1,0) & $(w-5)(w^4-20w^3+131w^2-342w+300)$
& $-2(w-5)^3(2w^2-12w+15)$ \\ \hline
\end{tabular}
\vspace{0.3 cm}\\
\begin{tabular}{|c|c|}
\hline
$I$ & $c_I$  \\ \hline
(4,0,0,0,0) (3,1,0,0,0) & $(w-2)(w-3)^2(w-4)^2$ \\ \hline
(0,4,0,0,0) & $(w-3)(w-4)(w-6)$ \\ \hline
(1,3,0,0,0) & $(w-2)(w-3)(w-4)(w-6)$ \\ \hline
(2,2,0,0,0) (2,0,1,0,0) & $(w-2)(w-3)^2(w-4)(w-6)$ \\ \hline
\end{tabular}
\begin{tabular}{|c|c|}
\hline
$I$ & $c_I$  \\ \hline
(0,2,1,0,0) (0,0,0,0,1) & $(w-4)(w-6)$\\ \hline
(1,1,1,0,0) & $(w-2)(w-4)(w-6)$ \\ \hline
(0,0,2,0,0) (0,1,0,1,0) & $(w-3)(w-4)(w-6)$ \\ \hline
(1,0,0,1,0) & $(w-2)(w-3)(w-4)(w-6)$ \\ \hline
\end{tabular}
\normalsize
\vspace{0.3 cm}\\

We have also computed the case $r=5$ and $r=6$ but these tables cannot be given here. They are available upon request. We recall that for $1 \le k\le w/2-2$ the coefficients $f^{(k)}_{\rho}$ and $f_I$ depend on $t$ and are not yet explicitly known.

\subsection{Complete functions}

An important application is obtained by specializing $z=0$. Since $P_k(0)=h_k$, this particular case corresponds to complete functions. Then $\phi_\rho(t)$ is the generating function of the coefficients $c_{\rho}^{(k)}$ of Theorem 7.1.

Denoting $|\rho|=w$ and
$R_k=k^{k}/k!$, the first terms of the expansion of $\phi_\rho(t)$ are
\begin{align*}
w! \,\phi_{\rho}(t)&= R_{w-1}(e^{(w-1)t}\pm e^{-(w-1)t})
-2R_{w-2}(m_1+w-2)(e^{(w-2)t}\pm e^{-(w-2)t})\\
&+\frac{R_{w-3}}{w-2}(e^{(w-3)t}\pm e^{-(w-3)t})
\Big((2m_1+w-3)(w-3)(2w-3)+(5w-9)\binom{m_1}{2}
\\&+(w-3)^2m_2\Big) -2\frac{R_{w-4}}{w-3}(e^{(w-4)t}\pm e^{-(w-4)t}) \Big(\frac{1}{3}(3m_1+w-4)(w-4)^2(2w-3)\\
&+(w-4)(5w-8)\binom{m_1}{2}
+(7w-16)\binom{m_1}{3}
+(w-4)^2(m_1m_2+(w-2)m_2-m_3)\Big)\\
&+R_{w-5}(e^{(w-5)t}\pm e^{-(w-5)t})\sum_{I\subset \mathbf{I}_4}
M_I(w,\mathbf{m})\frac{b_I(w)}{c_I(w)} + \mathrm{etc}\ldots.
\end{align*}
Here $w\ge 7$ is implicitly assumed. The coefficients $b_I(w)$, $c_I(w)$ are listed in the tables of Section 9.5.

For $w/2-2 < k\le w-1$, the coefficient of $e^{kt}\pm e^{-kt}$ does not depend on $t$. However for $1 \le k\le w/2-2$ it does depend on $t$ and is not yet explicitly known.

\section{Other symmetric functions}

Our method may be used for more symmetric functions than those presented above. In particular the results for the Hall-Littlewood symmetric functions may be immediately translated to three families: the Schur functions associated with hooks $s_{(a,1^b)}$, the partial sums $p_{a,b}=\sum_{|\la|=a,l(\la)=b}m_\la$, and the products $h_ae_b=s_{(a,1^{b})}+s_{(a+1,1^{b-1})}$.

This is a consequence of the fact that the Hall-Littlewood symmetric function $P_{k}(z)$ is a generating function for such families. More precisely we have~\cite[Example 3.2.3, p. 214; Example 1.3.9, p. 47]{Ma},~\cite[p. 240]{La3}
\[
P_{k}(z)=\sum_{r=0}^{k-1} (-z)^r s_{(k-r,1^r)}=\sum_{l=1}^k (1-z)^{l-1} p_{k,l}=(1-z)^{-1} \sum_{r=0}^k (-z)^{r}h_{k-r}e_r.
\]

Of course our method may be applied directly to $s_{(a,1^b)}$, $h_ae_b$, or $p_{a,b}$ without using the Hall-Littlewood polynomial. Then the recurrence for the class expansion coefficients depends on two parameters $a,b$. Here we only give our equations (4.14), with $u_i=\la_i-i+1$. For the product $h_ae_b$ they are
\begin{align*}
\sum_{i=1}^{l(\la)+1} 
c_i(\la) \,(h_{a}e_b)(A_{\la^{(i)}})&=(h_{a}e_b)(A_{\la})+\sum_{i=1}^{l(\la)+1} 
c_i(\la) \,u_i\,(h_{a-1}e_b)(A_{\la^{(i)}}),\\
\sum_{i=1}^{l(\la)+1} 
c_i(\la) \,u_i\,(h_{a}e_b)(A_{\la^{(i)}})&=n(h_{a}e_{b-1})(A_{\la})+\sum_{i=1}^{l(\la)+1} 
c_i(\la) \,u_i^2\,(h_{a-1}e_b)(A_{\la^{(i)}}).
\end{align*}
For the Schur functions $s_{(a,1^b)}$ they are
\begin{align*}
\sum_{i=1}^{l(\la)+1} 
c_i(\la) \,s_{(a,1^b)}(A_{\la^{(i)}})&=s_{(a,1^b)}(A_{\la})+\sum_{i=1}^{l(\la)+1} c_i(\la) \,u_i\,s_{(a-1,1^b)}(A_{\la^{(i)}}),\\
\sum_{i=1}^{l(\la)+1} 
c_i(\la) \,u_i\,s_{(a,1^b)}(A_{\la^{(i)}})&=ns_{(a,1^{b-1})}(A_{\la})+\sum_{i=1}^{l(\la)+1} c_i(\la) \,u_i^2\,s_{(a-1,1^b)}(A_{\la^{(i)}}).
\end{align*}
And for the partial sums $p_{a,b}$ they are
\begin{align*}
\sum_{i=1}^{l(\la)+1} 
c_i(\la) \,p_{a,b}(A_{\la^{(i)}})&=p_{a,b}(A_{\la})+\sum_{i=1}^{l(\la)+1} c_i(\la) \,u_i\,p_{a-1,b}(A_{\la^{(i)}}),\\
\sum_{i=1}^{l(\la)+1} 
c_i(\la) \,u_i\,p_{a,b}(A_{\la^{(i)}})&=n(p_{a-1,b-1}-p_{a-1,b})(A_{\la})+\sum_{i=1}^{l(\la)+1} c_i(\la) \,u_i^2\,p_{a-1,b}(A_{\la^{(i)}}).
\end{align*}
We leave other steps, and the proof of the following result, to the reader (see also Proposition 8.4).
\begin{theo}
We have the class expansion
\[h_re_s(J_1,\ldots,J_n)=\sum_{\rho}c^{(r,s)}_\rho \, \binom{n-|\overline{\rho}|}{m_1(\rho)} \, C_{\tilde{\rho}},\]
where the coefficients $c^{(r,s)}_\rho$ are determined by the recurrence relations
\begin{equation*}
c^{(r,s)}_{\rho\cup (1)}=\sum_{u \ge 1} um_u(\rho) \,c^{(r-1,s)}_{\rho\setminus (u) \cup(u+1)},
\end{equation*}
\begin{equation*}
\begin{split}
\sum_{u \ge 1} um_u(\rho) \,c^{(r,s)}_{\rho\setminus (u) \cup(u+1)}&=
|\rho|\,(2c^{(r-1,s)}_\rho+c^{(r,s-1)}_\rho)+m_1(\rho)\,(c^{(r-1,s)}_{\rho\setminus (1)}+c^{(r,s-1)}_{\rho\setminus (1)})\\
&+\sum_{u,v \ge 1} uvm_u(\rho)(m_v(\rho)-\delta_{uv})\, 
c_{\rho \setminus (u,v)\cup (u+v+1)}^{(r-1,s)}\\
&+\sum_{u,v \ge 1} (u+v-1)m_{u+v-1}(\rho)\,
c_{\rho \setminus (u+v-1)\cup (u,v)}^{(r-1,s)}.
\end{split}
\end{equation*}
The leading terms
are obtained for $|\rho|-l(\rho)=r+s$ and $m_1(\rho)=0$. Their coefficients are
\[c_{\rho}^{(r,s)}=\sum_{
\begin{subarray}{c}0\le s_i\le \rho_i-1\\s_1+\cdots+s_{l(\rho)}=s\end{subarray}} 
\prod_{i= 1}^{l(\rho)}\frac{1}{\rho_i}\binom{\rho_i}{s_i}\binom{2\rho_i-2-s_i}{\rho_i-1}.\]
\end{theo}
The leading coefficients can be quickly obtained as a consequence of Proposition 8.3 and the following identity~\cite[equ. (4)]{La5}
\begin{equation*}
z\,C_r(z)=\frac{1}{r+1}\sum_{m=0}^r  
(z-1)^{m}\binom{r+1}{m}\binom{2r-m}{r}.
\end{equation*}
\noindent\textit{Examples:}  For $r=0$ we have $s_i=\rho_i-1$ for any $i$ and we recover Theorem 5.1. For $r=1$ we have $s_i=\rho_i-1$ for any $i$ but one, equal to $\rho_i-2$, and we recover the leading coefficient $a_{\rho}=\sum_i \binom{\rho_i}{2}$ given in Proposition 5.2. For $s=0$, all $s_i$'s are zero and we recover Proposition 7.3. For $s=1$, all $s_i$'s are zero, but one equal to 1, and the leading coefficient of $h_re_1$ is
$\sum_i \binom{2\rho_i-3}{\rho_i-1}\prod_{j\neq i}C_{\rho_j-1}$.

Unfortunately our method is not efficient with the one-row Macdonald symmetric function, nor with the products $e_\mu$, $p_\mu$, $h_\mu$. With the latter, two difficulties are quickly encountered. Firstly the computations become very messy. Secondly one needs to extend the results of Theorem 4.1 in order to express 
\[\sum_{i}
c_i(\la) \, (\la_i-i+1)^k\,\ta^{\la^{(i)}}_\mu \quad \textrm{for} \: k \ge 3.\] 
However for $l\le 3$ the products $p_kp_l$ and the monomial symmetric functions $m_{(k,l)}=p_kp_l-p_{k+l}$ may be handled without any new ingredient. Actually for $p_kp_l$ we have\\
\begin{align*}
\sum_{i=1}^{l(\la)+1} 
c_i(\la) \,(p_kp_l)(A_{\la^{(i)}})&=(p_kp_l+p_kf_l-p_{k-1}f_{l+1})(A_{\la})+\sum_{i=1}^{l(\la)+1} c_i(\la) \,u_i\,(p_{k-1}p_l)(A_{\la^{(i)}}),\\
\sum_{i=1}^{l(\la)+1} 
c_i(\la) \,u_i\,(p_kp_l)(A_{\la^{(i)}})&=(p_kf_{l+1}-p_{k-1}(f_{l+2}+np_l))(A_{\la})+\sum_{i=1}^{l(\la)+1} c_i(\la) \,u_i^2\,(p_{k-1}p_l)(A_{\la^{(i)}}),
\end{align*}
with $f_i$ defined by (3.4). Therefore using (3.6), a recurrence may be defined provided $f_l,f_{l+1},f_{l+2}$ do not involve products $p_ap_b$, i.e. for $l\le3$.

\section{Extension to Jack polynomials}

Our method has a very natural extension in the framework of Jack polynomials. This generalization will be developed elsewhere. Here we only present some results (omitting the proofs).

Let $\al$ be some positive real parameter and $\be=\al-1$. The family of Jack polynomials $J_{\la}(\alpha)$, indexed by partitions, forms a basis of the algebra of symmetric functions with rational coefficients in $\al$~\cite{Ma,S0}. We consider the transition matrix between this basis and the classical basis of power sums $p_{\mu}$, i.e. we write
\[J_{\la}(\alpha)=\sum_{|\mu|= |\la|} \theta^{\la}_{\mu}(\al) \,p_{\mu}.\]

As a consequence of the Frobenius formula (see the argument in the introduction of~\cite{La4}), the quantities $\theta^{\la}_{\mu}(\al)$ generalize the central characters, i.e. we have
$\theta^{\la}_{\mu}(1)=\theta^\la_\mu=n!\,z_\mu^{-1} \hat{\chi}^\la_\mu$. 

Given a partition $\la$, the $\al$-content of any node $(i,j) \in \la$ is defined as $j-1-(i-1)/\al$. We denote by $A_\la^{(\al)}=\left\{j-1-(i-1)/\al,\, (i,j) \in \la \right\}$ the finite alphabet of the $\al$-contents of $\la$.

Denote by $\mathbf{Q}[\alpha]$ the field of rational functions in $\al$. A polynomial in $r$ indeterminates $\la=(\la_1,\ldots,\la_r)$ with coefficients in $\mathbf{Q}[\alpha]$ is said to be ``shifted symmetric'' in $\la$ if it is symmetric in the $r$ ``shifted variables'' $\la_i-i/\alpha$. In analogy with symmetric functions, this defines $\mathbf{S}^{\ast}(\al)$, the algebra of shifted symmetric functions with coefficients in $\mathbf{Q}[\alpha]$. We refer to~\cite{OO,Ok1,Ok2}, or to~\cite{La1,La4} for a short survey.

It is known~\cite[Proposition 2]{La4} that the quantities $\theta^{\la}_{\mu}(\al)$ are shifted symmetric functions of $\la$, and form a basis of $\mathbf{S}^{\ast}(\al)$.  Moreover~\cite[Lemma 7.1]{La1}, given a symmetric function $f$, its $\al$-content evaluation $f(A_\la^{(\al)})$ is also a shifted symmetric function of $\la$. The argument is similar to the one already given in Section 2.4.

It is therefore a natural problem to consider the expansion
\[f(A_\la^{(\al)})=\sum_{|\mu|=n}  a_{\mu}(n) \, \ta^\la_\mu(\al),\]
with $|\la|=n$, and to study the properties of the coefficients $a_{\mu}(n)$.

The simplest result of this type is the following generalization of Jucys' classical result
\[\al^k e_k(A_\la^{(\al)})=\sum_{\begin{subarray}{c}|\mu|=n\\
l(\mu)=n-k \end{subarray}}  \ta^\la_\mu(\al).\]
This expansion was first obtained in~\cite[Theorem 5.4]{La6}, as a consequence of the ``Cauchy formula'' for Jack polynomials (see also~\cite[Prop. 8.3]{Mat}). Another proof may be obtained by generalizing the argument in Section 5, which also yields the expansion
\[\al^{k+1} (e_1e_k)(A_\la^{(\al)})=\sum_{\begin{subarray}{c}|\mu|=n\\
l(\mu)=n-k-1 \end{subarray}}  a_\mu \ta^\la_\mu(\al)+
\be \sum_{\begin{subarray}{c}|\mu|=n\\
l(\mu)=n-k \end{subarray}}  a_\mu \ta^\la_\mu(\al)
+\al \sum_{\begin{subarray}{c}|\mu|=n\\
l(\mu)=n-k+1 \end{subarray}} \left(\binom{n}{2}-a_\mu\right) \ta^\la_\mu(\al)\]
with $a_\mu$ defined in Proposition 5.2.

A second important case is the extension of Lascoux-Thibon's result, i.e. the expansion
\[\al^k p_k(A_\la^{(\al)})=\sum_{|\mu|=n}  a^{(k)}_{\mu}(n) \, \ta^\la_\mu(\al).\]
A generalization of the method in Section 6 provides the following result. 
\begin{theo}
In the previous expansion, the coefficients $a^{(k)}_{\mu}(n)$ are polynomials in $n$, written as
\[a^{(k)}_{\mu}(n)=\sum_{\overline{\rho}=\overline{\mu}}c^{(k)}_\rho \, \binom{n-|\overline{\rho}|}{m_1(\rho)}.\]
Here the quantities $c^{(k)}_\rho$ are polynomials in $(\al,\be)$ with nonnegative integer coefficients, determined by the recurrence relations
\begin{align*}
c^{(k)}_{\rho\cup (1)}&=\al \sum_{r \ge 1} rm_r(\rho) \,c^{(k-1)}_{\rho\setminus (r) \cup(r+1)},\\
\sum_{r \ge 1} rm_r(\rho) \,c^{(k)}_{\rho\setminus (r) \cup(r+1)}&=
|\rho|\,c^{(k-1)}_\rho
+\al \sum_{r,s \ge 1} rsm_r(\rho)(m_s(\rho)-\delta_{rs})\, 
c_{\rho \setminus (r,s)\cup (r+s+1)}^{(k-1)}\\
&+\sum_{r,s \ge 1} (r+s-1)m_{r+s-1}(\rho)\,
c_{\rho \setminus (r+s-1)\cup (r,s)}^{(k-1)}+\be\sum_{r\ge 1}r^2 m_r(\rho) c_{\rho \setminus (r)\cup (r+1)}^{(k-1)}.
\end{align*}
\end{theo}
By induction on $k$ and the lowest part of $\rho$, the polynomials $c_{\rho}^{(k)}$ are non zero for $|\rho|+l(\rho)\le k+2$. Their generating function 
\[\phi_\rho(t)=\sum_{k\ge 0} c_{\rho}^{(k)} \, \frac{t^k}{k!}\]
can be determined. However the situation is much more intricate than for $\al=1$. In particular $\phi_\rho$ cannot be written in factorized form. 

The first values are given by
\begin{align*}
\phi_2(t)&=\frac{e^{\al t}-e^{-t}}{\al+1}, \qquad 
\phi_{1^2}(t)=\frac{e^{\al t}+\al e^{-t}}{\al+1}-1,\\
\phi_3(t)&=\frac{e^{2\al t}-e^{-t}}{(\al+1)(2\al+1)}
-\frac{e^{\al t}-e^{-2t}}{(\al+1)(\al+2)},\\
\phi_{21}(t)&=\frac{e^{2\al t}+2\al e^{-t}}{(\al+1)(2\al+1)}
-\frac{2e^{\al t}+\al e^{-2t}}{(\al+1)(\al+2)},\\
\phi_{1^3}(t)&=\frac{e^{2\al t}-4\al^2 e^{-t}}{(\al+1)(2\al+1)}
-\frac{4e^{\al t}-\al^2 e^{-2t}}{(\al+1)(\al+2)}+1.
\end{align*}

Generalizing the method of Section 8, we have similar results for the Hall-Littlewood symmetric function $P_k(z)$. Its $\al$-content evaluation may be written as 
\[\al^k P_k(A_\la^{(\al)};z)=\sum_{\rho}c^{(k)}_\rho \, \binom{n-|\overline{\rho}|}{m_1(\rho)} \,\ta^\la_{\tilde{\rho}}(\al),\]
where the coefficients $c^{(k)}_\rho$ are polynomials in $(\al,\be)$, which are nonzero for $|\rho|-l(\rho)\le k$. We list them below for $k\le 4$. The values for $h_k$ and $p_k$ are obtained for $z=0$ and $z=1$.
\footnotesize
\vspace{0.3 cm}\\
\begin{tabular}{|c|c|} \hline
$\rho$ & 2 \\ \hline
$c_{\rho}^{(1)}$ & 1 \\ \hline
\end{tabular}
\hspace{1cm}
\begin{tabular}{|c|c|c|c|c|} \hline
$\rho$ & 3 & $2^2$ & 2 & $1^2$ \\ \hline
$c_{\rho}^{(2)}$ & $2-z$ & $1-z$ & $\be$ & $\al$\\ \hline
\end{tabular}
\vspace{0.3cm}\\
\begin{tabular}{|c|c|c|c|c|c|c|c|c|c|} \hline
$\rho$ & 4 & 32 & $2^3$ & $2^2$ & 3 &$21^2$ & 21 & 2 & $1^2$\\ \hline
$c_{\rho}^{(3)}$ & $z^2-5z+5$ & $(1-z)(2-z)$ & $(1-z)^2$ & $2\be(1-z)$ & $3\be(2-z)$ & $\al(1-z)$ & $2\al(2-z)$ & $\al+\be^2$ & $\al\be$\\ \hline
\end{tabular}\\
\vspace{0.3cm}\\
\begin{tabular}{|c|c|c|c|c|c|c|}
\hline
$\rho$ & 5 & 42 & $3^2$ & $32^2$ & $2^4$  \\ \hline
$c_{\rho}^{(4)}$ & $(2-z)(z^2-7z+7)$ & $(1-z)(z^2-5z+5)$ & $(1-z)(2-z)^2$ & $(1-z)^2(2-z)$ & $(1-z)^3$   \\ \hline
\end{tabular}\\
\begin{tabular}{c|c|c|c|c|c|}
\rule{.55 cm}{0cm}& 4 & 32 & $2^3$  & $31^2$ & 31  \\ \cline{2-6}
& $\be(6z^2-29z+29)$ & $4\be(1-z)(2-z)$ & $3\be(1-z)^2$ & $\al(1-z)(2-z)$ & $3\al(z^2-5z+5)$\\ \cline{2-6}
\end{tabular}\\
\begin{tabular}{c|c|c|c|c|} 
\rule{.55 cm}{0cm}& 3 & $2^21^2$ & $2^21$ & $2^2$  \\ \cline{2-5}
& $(5\al+7\be^2)(2-z)$ & $\al(1-z)^2$ & $4\al(1-z)(2-z)$ & $4\al(z^2-5z+5)+3\be^2(1-z)$ \\ \cline{2-5}
\end{tabular}\\
\begin{tabular}{c|c|c|c|c|c|c|}
\rule{.55 cm}{0cm}& $21^2$ & 21 & 2 & $1^4$ & $1^3$ & $1^2$  \\ \cline{2-7}
& $2\al\be(1-z)$ & $6\al\be(2-z)$ & $2\al\be+\be^3$ & $3\al^2(1-z)$ & $4\al^2(2-z)$ & $\al^2+\al\be^2$ \\ \cline{2-7}
\end{tabular}
\vspace{0.3 cm}\\
\normalsize

However we emphasize that, given a symmetric function $f$, we are as yet unable to translate its $\al$-content expansion $f(A_\la^{(\al)})$ in terms of the specialization of $f$ at some generalized Jucys-Murphy elements. Actually, at this moment, we do not know how the symmetric algebra and the Jucys-Murphy elements might be generalized for $\al \neq 1$. 

The only known exception is for $\al=2$ and $\al=1/2$, where a deep interpretation has been recently found by Matsumoto~\cite{Mat} in terms of odd Jucys-Murphy elements $(J_1,J_3,\ldots, J_{2n-1})$ of $S_{2n}$.

\section{Appendix}

\begin{lem}
Let $z$ be an indeterminate. The quantities
\begin{equation}
a^{(k)}_{\mu}(n)=\sum_{\overline{\rho}=\overline{\mu}}c^{(k)}_\rho \, \binom{n-|\overline{\mu}|}{m_1(\rho)}
\end{equation}
satisfy the recurrence relations
\begin{equation}
a^{(k)}_{\mu\cup(1)}(n+1)=a^{(k)}_\mu(n)
+\sum _{r\ge 1}rm_r(\mu)\,a^{(k-1)}_{\mu \setminus (r)\cup (r+1)}(n+1),
\end{equation}
\begin{align}
\sum_{r\ge 1}rm_r(\mu)\,a^{(k)}_{\mu \setminus (r)\cup (r+1)}(n+1)
&=-nza^{(k-1)}_{\mu}(n)\notag\\
&+\sum_{r,s \ge 1} rsm_r(\mu)(m_s(\mu)-\delta_{rs})\, 
a_{\mu \setminus (r,s)\cup (r+s+1)}^{(k-1)}(n+1)
\\ &+ \sum_{r,s \ge 1} (r+s-1)m_{r+s-1}(\mu)\,
a_{\mu \setminus (r+s-1)\cup (r,s)}^{(k-1)}(n+1).\notag
\end{align}
if and only if the coefficients $c^{(k)}_\rho$ satisfy the recurrence relations
\begin{align}
c^{(k)}_{\rho\cup (1)}&=\sum_{r \ge 1} rm_r(\rho) \,c^{(k-1)}_{\rho\setminus (r) \cup(r+1)},\\
\sum_{r \ge 1} rm_r(\rho) \,c^{(k)}_{\rho\setminus (r) \cup(r+1)}&=
(2-z)|\rho|\,c^{(k-1)}_\rho +(1-z)m_1(\rho)\,c^{(k-1)}_{\rho\setminus (1)}\notag\\
&+\sum_{r,s \ge 1} rsm_r(\rho)(m_s(\rho)-\delta_{rs})\, 
c_{\rho \setminus (r,s)\cup (r+s+1)}^{(k-1)}\\
&+\sum_{r,s \ge 1} (r+s-1)m_{r+s-1}(\rho)\,
c_{\rho \setminus (r+s-1)\cup (r,s)}^{(k-1)}.\notag
\end{align}
\end{lem}
\begin{proof} When substituting (12.1) into (12.2)--(12.3) we must distinguish the parts 1 of $\mu$ since $m_1(\mu)=n-|\overline{\mu}|$ depends on $n$. Firstly (12.1) yields
\[a^{(k)}_{\mu\cup (1)}(n+1)-a^{(k)}_{\mu}(n)=\sum_{\overline{\rho}=\overline{\mu}}c^{(k)}_\rho \, \binom{n-|\overline{\mu}|}{m_1(\rho)-1},\]
so that (12.2) may be written as
\begin{align*}
\sum_{\overline{\rho}=\overline{\mu}}c^{(k)}_\rho \, \binom{n-|\overline{\mu}|}{m_1(\rho)-1}&=
(n-|\overline{\mu}|)
\sum_{\overline{\sigma}=\overline{\mu}\cup (2)}c^{(k-1)}_{\sigma} \, \binom{n-|\overline{\mu}|-1}{m_1(\sigma)}\\
&+\sum _{r\ge 2}
rm_r(\mu)\,\sum_{\overline{\tau}=\overline{\mu}\setminus (r)\cup (r+1)}c^{(k-1)}_{\tau} \, \binom{n-|\overline{\mu}|}{m_1(\tau)}.
\end{align*}
By identification of the coefficients of $\binom{n-|\overline{\mu}|}{m_1(\rho)}$ on both sides, (12.4) follows. 

Secondly we have
\begin{align*}
\sum_{\{r=1\} \cup \{s=1\}} (r+s-1)m_{r+s-1}(\mu)\,
a_{\mu \setminus (r+s-1)\cup (r,s)}^{(k-1)}(n+1)&=
(2n-m_1(\mu))\,a_{\mu \cup (1)}^{(k-1)}(n+1)\\
&=(n+|\overline{\mu}|)\sum_{\overline{\tau}=\overline{\mu}}c^{(k-1)}_{\tau} \, \binom{n-|\overline{\mu}|+1}{m_1(\tau)},
\end{align*}
and (12.3) may be written as
\begin{align*}
(n&-|\overline{\mu}|)
\sum_{\overline{\sigma}=\overline{\mu}\cup (2)}c^{(k)}_{\sigma} \, \binom{n-|\overline{\mu}|-1}{m_1(\sigma)}
+\sum _{r\ge 2}
rm_r(\mu)\,\sum_{\overline{\tau}=\overline{\mu}\setminus (r)\cup (r+1)}c^{(k)}_{\tau} \, \binom{n-|\overline{\mu}|}{m_1(\tau)}\\
=&-nz\sum_{\overline{\rho}=\overline{\mu}}c^{(k-1)}_\rho \, \binom{n-|\overline{\mu}|}{m_1(\rho)}+ (n+|\overline{\mu}|)\sum_{\overline{\rho}=\overline{\mu}}c^{(k-1)}_{\rho} \, \binom{n-|\overline{\mu}|+1}{m_1(\rho)}\\
&+ \sum_{r,s \ge 2} (r+s-1)m_{r+s-1}(\mu)\,
\sum_{\overline{\tau}=\overline{\mu}\setminus (r+s-1)\cup (r,s)}c^{(k-1)}_{\tau} \, \binom{n-|\overline{\mu}|}{m_1(\tau)}\\
&+\sum_{r,s \ge 2} rsm_r(\mu)(m_s(\mu)-\delta_{rs})\, 
\sum_{\overline{\sigma}=\overline{\mu}\setminus (r,s)\cup (r+s+1)}c^{(k-1)}_{\sigma} \, \binom{n-|\overline{\mu}|}{m_1(\sigma)}\\
&+2 (n-|\overline{\mu}|)\sum_{r \ge 2} rm_r(\mu)\, 
\sum_{\overline{\sigma}=\overline{\mu}\setminus (r)\cup (r+2)}c^{(k-1)}_{\sigma} \, \binom{n-|\overline{\mu}|-1}{m_1(\sigma)}\\
&+(n-|\overline{\mu}|)(n-|\overline{\mu}|-1)\, 
\sum_{\overline{\sigma}=\overline{\mu}\cup (3)}c^{(k-1)}_{\sigma} \, \binom{n-|\overline{\mu}|-2}{m_1(\sigma)}.
\end{align*}
But we have the identity
\begin{align*}
(n+a)\binom{n-a+1}{b}-nz\binom{n-a}{b}&=(1-z)(b+1)\binom{n-a}{b+1}+(2-z)(a+b)\binom{n-a}{b}\\&+(2a+b-1)\binom{n-a}{b-1},
\end{align*}
which we apply with $a=|\overline{\mu}|=|\overline{\rho}|$ and $b=m_1(\rho)$ so that $a+b=|\rho|$ and $2a+b=2|\rho|-m_1(\rho)$. Now (12.5) follows by identifying the coefficients of $\binom{n-|\overline{\mu}|}{m_1(\rho)}$ on both sides and using
\[
\sum_{\{r=1\} \cup \{s=1\}}(r+s-1)m_{r+s-1}(\rho)\,
c_{\rho \setminus (r+s-1)\cup (r,s)}^{(k-1)}=(2|\rho|-m_1(\rho))c_{\rho \cup (1)}^{(k-1)}. \]
\end{proof}

\end{document}